\documentclass[a4paper,twoside,english]{amsart}

\usepackage[T1]{fontenc}
\usepackage[latin9]{inputenc}
\usepackage{babel}
\usepackage{verbatim}
\usepackage{amsthm}
\usepackage{amsmath}
\usepackage{amssymb}
\usepackage{xcolor}
\usepackage{mathtools}
\usepackage{hyperref}

\hypersetup{
 pdfauthor={L. Capuano, F. Veneziano and U. Zannier}}

\usepackage{algorithm,algorithmic}

\makeatletter

\pdfpageheight\paperheight 
\pdfpagewidth\paperwidth

  \theoremstyle{plain}
  \newtheorem{thm}{\protect\theoremname}[section]
  \theoremstyle{remark}
  \newtheorem{rem}[thm]{\protect\remarkname}
  \theoremstyle{definition}
  \newtheorem*{example*}{\protect\examplename}
  \theoremstyle{definition}
  \newtheorem{example}[thm]{\protect\examplename}
  \theoremstyle{plain}
  \newtheorem{lem}[thm]{\protect\lemmaname}
  \theoremstyle{plain}
  \newtheorem{prop}[thm]{\protect\propositionname}
  \theoremstyle{plain}
  \newtheorem{cor}[thm]{\protect\corollaryname}
  \theoremstyle{definition}
  \newtheorem{definition}[thm]{\protect\definitionname}

 \numberwithin{equation}{section}
 
  \providecommand{\definitionname}{Definition}
  \providecommand{\examplename}{Example}
  \providecommand{\lemmaname}{Lemma}
  \providecommand{\propositionname}{Proposition}
  \providecommand{\remarkname}{Remark}
  \providecommand{\theoremname}{Theorem}
\providecommand{\theoremname}{Theorem}
 \providecommand{\corollaryname}{Corollary}

\usepackage[left=3cm,right=3cm,top=3.3cm,bottom=3.3cm]{geometry}

\def\N{{\mathbb N}}
\def\Q{{\mathbb Q}}
\def\R{{\mathbb R}}

\def\Z{{\mathbb Z}}

\def\A{{\mathbb A}}

\def\p{{\bf p}}
\def\q{{\bf q}}
\def\x{{\bf x}}

\def\vl{{v}}

\def\Ql{{\mathbb Q_{\ell}}}

\def\CVD{{\hfill\hfil{\lower 2pt\hbox{\vrule\vbox to 7pt
{\hrule width  5pt\varphifill\hrule}\varphirule}}}\par}

\newcommand{\craz}{C_1}
\newcommand{\ceigen}{C_2}
\newcommand{\ceigensnd}{C_3}
\newcommand{\ceigentrd}{C_4}

\newcommand{\abs}[1]{\left| #1 \right|}
\newcommand{\floor}[1]{\left\lfloor #1 \right\rfloor}
\newcommand{\pfloor}[1]{\left\lfloor #1 \right\rfloor_{\ell}}
\newcommand{\sd}{\sqrt{|\Delta|}}

\title[Periodicity of $\ell$-adic continued fractions]{An effective criterion for periodicity of $\ell$-adic continued fractions}

\author[L. Capuano]{Laura Capuano}

\address{Mathematical Institute \\
University of Oxford \\
Oxford OX2 6GG \\
UK}
\email{Laura.Capuano@maths.ox.ac.uk}

\author[F. Veneziano]{Francesco Veneziano}

\address{Centro di Ricerca Matematica Ennio De Giorgi \\
Piazza dei Cavalieri, 3, 56126 Pisa\\ 
Italy}
\email{francesco.veneziano@sns.it}

\author[U. Zannier]{Umberto Zannier}

\address{Scuola Normale Superiore\\
Piazza dei Cavalieri 7, 56126 Pisa\\
Italy}
\email{u.zannier@sns.it}


\begin{document}

\begin{abstract}
The theory of continued fractions has been generalized to $\ell$-adic numbers by several authors and presents many differences with respect to the real case. In the present paper we investigate the expansion of rationals and quadratic irrationals for the $\ell$-adic continued fractions introduced by Ruban. In this case, rational numbers may have a periodic non-terminating continued fraction expansion; moreover, for quadratic irrational numbers, no analogue of Lagrange's theorem holds. 
We give general explicit criteria to establish the periodicity of the expansion in both the rational and the quadratic case (for rationals, the qualitative result is due to Laohakosol \cite{Laohakosol85}).  
\end{abstract}

\maketitle

\section{Introduction}
The theory of real continued fractions plays a central role in real Diophantine Approximation for many different reasons, in particular because the convergents of the simple continued fraction expansion of a real number $\alpha$ give the best rational approximations to $\alpha$. 
Motivated by the same type of questions, several authors (Mahler \cite{Mahler34}, Schneider \cite{Schneider70},  Ruban \cite{Ruban70}, Bundschuh \cite{Bundschuh77} and
Browkin \cite{Browkin78}) have generalized the theory of real continued fractions to the $\ell$-adic case in various ways.

\medskip

There is no canonical way to define a continued fraction expansion in this context, as we lack a canonical $\ell$-adic analogue of the integral part. The $\ell$-adic process which is the most similar to the classical real one was mentioned for the first time in one of the earliest papers on the subject by Mahler \cite{Mahler34}, and then studied accurately by Ruban \cite{Ruban70}, who showed that these continued fractions enjoy nice ergodic properties. 

Ruban's continued fractions will be the subject of this paper and they have many important differences with respect to the classical real ones. First of all, while some rational numbers have a finite expansion, this is not\,---\,unlike the real case\,---\,the only possible behaviour. For example, it is easy to see that negative rational numbers cannot admit a terminating Ruban continued fraction. 

It is however possible to decide when a given rational number admits a finite Ruban continued fraction expansion and indeed our first result is the following:

\begin{thm}\label{T.terminating} Let $\ell$ be a prime number and $\alpha\in\Q$ be a rational number.  
\begin{enumerate}
    \item[(i)] The Ruban continued fraction expansion of $\alpha$ terminates if and only if all complete quotients are non-negative; there is an algorithm to decide in a finite number of steps whether this happens.
 \item[(ii)] If the Ruban continued fraction expansion of $\alpha$ does not terminate, then it is periodic with all partial quotients eventually   equal to $\ell -\ell^{-1}$; in this case, the 
pre-periodic part can be effectively computed.
\end{enumerate}
\end{thm}

Disregarding the computability aspect, the last part of this  result has already appeared in the literature, due to Laohakosol and, independently, to Wang (see \cite{Laohakosol85} and \cite{Wang85}), but this does not seem to be the case for either of the algorithmic conclusions, which apparently do not follow directly from the proofs in \cite{Laohakosol85} and \cite{Wang85}. For completeness, we have also included our own (short) proof of the qualitative part, which is quite  different.

The conclusion of Theorem~\ref{T.terminating} depends of course on the precise algorithm defining the continued fraction expansion. In \cite{Browkin78}, Browkin modified Ruban's definition so that every rational number has a finite $\ell$-adic continued fraction expansion.

\medskip

Another natural question arises when one considers the periodicity of Ruban continued fractions. In the classical real case, Lagrange's theorem states that a real number has an infinite periodic continued fraction if and only if it is quadratic irrational.  We will show that this is not true in the $\ell$-adic case and only some similarities can be recovered. For example, in Section~\ref{S.purely} we will prove the following result:

\begin{thm}\label{T.periodic}
An element $\alpha\in\Q_\ell$ whose Ruban continued fraction expansion is periodic is either rational or quadratic irrational over $\Q$ and such that $\Q(\alpha)$ can be embedded in $\R$. 
\end{thm}

No full analogue of Lagrange's theorem holds in this setting as remarked by Ooto in \cite{Ooto17}, and the problem of deciding whether a quadratic ($\ell$-adic) irrational number has a periodic continued fraction seems still open. For Browkin's definition, some very partial sufficient conditions were given in a series of papers by Bedocchi \cite{Bedocchi93}. Moreover, in \cite{Browkin01}, Browkin wrote an algorithm to generate the periodic continued fraction expansion of $\sqrt{\Delta}\in \Q_{\ell}\setminus \Q$ for some values of $\Delta$ and $\ell$, giving many numerical examples. 

In this paper we investigate the periodicity of the $\ell$-adic Ruban continued fraction expansion of quadratic irrational numbers, thus solving a problem posed by Laohakosol in \cite{Laohakosol85}.

Our main result is the following:

\begin{thm} 
Let $\alpha\in \Q_{\ell}\setminus \Q$ be a quadratic irrational over $\Q$. Then, the Ruban continued fraction expansion of $\alpha$ is periodic if and only if there exists a unique real embedding $j:\Q(\alpha) \rightarrow \R$ such that the image of each complete quotient $\alpha_n$ under the map $j$ is positive.

Moreover, there is an effectively computable constant $N_{\alpha}$ with the property that, either $\exists~ n\leq N_\alpha$ such that $\alpha_n$ does not have a positive real embedding, and therefore the expansion is not periodic, or $\exists~ n_1<n_2\leq N_\alpha$ such that $\alpha_{n_1}=\alpha_{n_2}$, hence the expansion is periodic.

In particular, both the preperiodic and the periodic part of a periodic expansion can be computed with a finite algorithm.
\end{thm}

This theorem follows directly from Theorem~\ref{general_condition} and Lemma~\ref{remark.negativi.restano.negativi} in Section~\ref{S.periodic}. A suitable constant $N_{\alpha}$ is explicitly computed in the proof of Theorem~\ref{general_condition} (see also Remark~\ref{explicit_constant} for an optimized value of it).

\medskip

It is also interesting to study how the qualitative behaviour of the expansion varies with the prime $\ell$ for fixed rational or irrational quadratic numbers. We will show that finiteness of the expansion (for rational numbers) and periodicity (for irrational quadratic numbers) are ``unlikely'' behaviours which occur for at most finitely many primes.

\medskip

In Section~\ref{algorithms} we carry out a formal complexity analysis of the effective algorithms outlined in the paper.


\section{General properties of the continued fractions in \texorpdfstring{$\Q_{\ell}$}{Ql}}

For the rest of the paper, we will denote by $\ell$ a prime number, and we will consider continued fractions in $\Q_{\ell}$. By $v=v_{\ell}$, we will mean the usual $\ell$-adic valuation, and similarly for $\abs{\cdot}_{\ell}$. For any $\alpha$ algebraic over $\overline{\Q}$ we will denote by $h(\alpha)$ the absolute logarithmic height and by $H(\alpha)$ the multiplicative height (for a precise definition see \cite[Section 1.5]{BG}).

\bigskip

A \textit{simple continued fraction} is an expression (either finite or infinite) of the form
\begin{equation*}
a_0+\cfrac{1}{a_1+\cfrac{1}{a_2+\cdots }}.
\end{equation*}
The $a_i$ for $i=0,1,2 \ldots$ may be in a fixed field, and they are usually taken in $\Z$ when the field is $\R$ and in $\Q$ when the field is $\Q_{\ell}$. The terms $a_0,a_1, \ldots$ are called \textit{partial quotients}.

\medskip

The standard notation for a simple continued fraction is $[a_0,a_1, \ldots]$. If the partial quotients become periodic after some point, then a bar is placed over the repeating partial quotients, \textit{i.e.} $[a_0,\ldots, a_{n-1},\overline{a_{n},\ldots, a_{n+k-1}}]$. In this case, the minimal $k$ is called the \textit{length} of the period. We call the continued fraction \textit{purely periodic} if the periodic part starts with $a_0$.

\medskip

Starting with a real number $\alpha$, its classical continued fraction expansion is given by the following algorithm:  we define $a_0:=\lfloor \alpha \rfloor$, 
$r_0:=\alpha-a_0$, and, by recurrence, $a_{n+1}:=\lfloor 1/r_n \rfloor$ and $r_{n+1}:=r_n^{-1}-a_{n+1}$, whenever $r_n \neq 0$. If on the other hand $r_n=0$ for some $n$, then the procedure stops.
 With this definition, we have that for $i\ge 1$ the $a_i$ are all positive integers, while $a_0$ has the same sign of the starting $\alpha$. Moreover, the procedure eventually stops if and only if we start with a rational number.
This construction leads to the best rational approximations to real numbers, and produces an eventually periodic expansion for irrational quadratic numbers.
Indeed a famous theorem of Lagrange says that the simple continued fraction expansion of a real number $\alpha$ is periodic if and only if $\alpha$ is quadratic irrational.

\medskip

For a real number, the algorithm used to construct the continued fraction is well-defined, since for every $\alpha$ there is only one integer $a$ such that $0\le \alpha-a <1$. If $\alpha \in \Q_{\ell}$ instead, there are infinitely many $ a\in \Z$ such that $0\le |\alpha-a|_{\ell}<1$, and there is no canonical choice nor any obvious way of choosing $a$ so that the analogues of theorems about real continued fractions hold. Many authors (see \cite{Schneider70}, \cite{Ruban70}, \cite{Browkin78}) gave different definitions of $\ell$-adic continued fractions.
In the rest of the paper, we will focus on the definition given by Ruban \cite{Ruban70}. We will refer to Ruban's definition using the abbreviation RCF.

\bigskip
Notice that, in oreder to have a simple continued fraction we need to take $a\in \Z\left[\frac{1}{\ell}\right]$ with $v(a)=v(\alpha)$.
Following Ruban, we give the following definition:
\begin{definition} \label{integral_part}
  The $\ell$-adic integral part $\lfloor \alpha\rfloor_{\ell}$  of $\alpha\in\Q_{\ell}$  is the unique  
  $a\in\Z\left[\frac{1}{\ell}\right]$ such that $0\le a <\ell$ and $|\alpha-a|_{\ell}< 1$. 
  \end{definition}

  Given this definition, we may expand elements in $\Q_{\ell}$  in a continued fraction in the usual way: put $\alpha_0:=\alpha$, $a_0:=\lfloor\alpha_0\rfloor_{\ell}$ and $r_0:=\alpha_0-a_0$; by definition, $|r_0|_{\ell}<1$. Then for all $n\geq 0$ define by recurrence:
    \begin{equation*} \label{recurrence}
    \alpha_{n+1}:=1/r_n, \quad a_{n+1}:=\lfloor\alpha_{n+1}\rfloor_{\ell} \quad \mbox{and} \quad r_{n+1}:=\alpha_{n+1}-a_{n+1},
    \end{equation*}
whenever $r_n\neq 0$; if on the other hand $r_n=0$ for some $n$, then the procedure stops. Thus a given number $\alpha\in\Q_\ell$ uniquely defines a RCF expansion. We also define $e_n:=-v(\alpha_n)$ for $n\geq 0$.  
\medskip

Notice that, if $|\alpha|_{\ell}<1$, then $a_0=0$ and $\alpha_1=\frac{1}{\alpha}$ with $|\alpha_1|_{\ell}>1$; so, shifting $n$ by one if necessary, we can always assume that $|\alpha|_{\ell}\ge 1$, so that $a_0\neq 0$.

The partial quotients $a_n$ so obtained satisfy, for $n>0$,
    \begin{equation*} \label{a.conditions}
    |a_n|_{\ell}={\ell}^{e_n}, \qquad e_n>0, \qquad  0<a_n\le \ell-\ell^{-e_n}<\ell .
    \end{equation*}

  As in the classical real case, we can define two sequences $(p_n)$ and $(q_n)$ by setting\footnote{Several authors use shifted indices, so that their $p_n$ is our $p_{n+1}$.} 
    \begin{equation} \label{E.convergents}
    \begin{cases}
    p_n,q_n=0  \ \ &\forall\, n< 0; \\
    p_0=1, \quad p_1=a_0, \\
    p_{n+1}=a_np_n+p_{n-1} \  &\forall\, n\ge 1; \\
    q_0=0, \quad q_1=1, \\ 
    q_{n+1}=a_nq_n+q_{n-1} \ &\forall\, n\ge 1.
    \end{cases}
    \end{equation}
  Note that the $p_n,q_n$ are rational numbers whose reduced denominator is a power of $\ell$. They are non-negative and they satisfy the usual formula
   \begin{equation} \label{E.convergents2}
   p_nq_{n-1}-p_{n-1}q_n=(-1)^n.
   \end{equation}
   
  We call  $(p_n/q_n)$  the  \textit{convergents} of the continued fraction; for $n>0$, we have
   \begin{equation*} 
   \frac{p_n}{q_n}=[a_0,a_1, \ldots, a_{n-1}].
   \end{equation*}
   
  The $\alpha_n$ which appear in the algorithm are called \textit{complete quotients} and they satisfy, as in the classical case, several relations with the $p_{n}$ and $q_n$ (see \cite{HW} for a classical account). 
  In particular we have, for $n>0$,
   \begin{align} 
 \label{E.cq}  \alpha&=\frac{\alpha_n p_n+p_{n-1}}{\alpha_n q_n+q_{n-1}},\\
\label{E.approx} \varphi_n&:= p_n-\alpha q_n={\frac{(-1)^{n}}{\alpha_nq_n+q_{n-1}}}.
  \end{align} 
   

\subsection{Convergence in \texorpdfstring{$\Q_\ell$}{Ql}}
Notice that, as $|a_i|_{\ell}\ge \ell >1$ for $i\le 1$, the continued fraction converges in $\Q_{\ell}$. In fact, if $v(a_n)=-e_n$, we have by induction that
   \begin{equation*}\label{orders}
   -v(p_n)=e_0+\ldots +e_{n-1},\qquad -v(q_n)=e_1+\ldots +e_{n-1},
   \end{equation*}
for $n\ge 1$. Then, equation~\eqref{E.convergents2}  gives $v((p_n/q_n)-(p_{n-1}/q_{n-1}))= -v(q_n)-v(q_{n-1})\ge 2n-3$, proving in particular that $(p_n/q_n)_{n\in\N}$ is a Cauchy sequence in $\Q$ with respect to the $\ell$-adic metric.

Note that the limit, when expanded 
 in a RCF,  gives back  the same continued fraction, since the definition of {\it $\ell$-adic integral part}, as given above, is well-posed.
 
 We shall denote by $[a_0,\dotsc]_{\ell}$ the value of a continued fraction in $\Q_{\ell}$ defined as the limit, whose existence has just been proved, in the $\ell$-adic metric. Moreover, we have  $\alpha_n=[a_n,a_{n+1},\dotsc]_\ell$ for all $n\ge 0$.
 
\medskip

As a matter of notation, we shall put 
\begin{equation*}
s_{n}:=e_0+ \cdots+ e_{n-1}\qquad\hbox{for every $n\ge 1$}. 
\end{equation*}
Notice that $s_n\ge n-1$, as $e_i\ge 1$ for all $i\ge 1$. In this way, we have $v(p_n)=-s_n$ and $v(q_n)=e_0-s_n$.

\bigskip

From the construction of the RCF expansion we have $v(\alpha_n)=v(a_n)$ and, inserting in~\eqref{E.approx}, we find 
   \begin{equation}\label{E.approx2} 
   v(p_n-\alpha q_n)=-v(q_{n+1})=s_{n+1}-e_0\ge n,
   \end{equation}
   unless the continued fraction stops and $p_n=\alpha q_n$.

  Inverting~\eqref{E.cq} we also find
  \begin{equation} \label{E.cq2}
   \alpha_n=-\frac{\varphi_{n-1}}{\varphi_n}=\frac{p_{n-1}-\alpha q_{n-1}}{\alpha q_n-p_n}\, .
   \end{equation}

We will also define, for $n\ge 1$,
\begin{equation*}\label{E.tilde}
\tilde p_n:=\ell^{s_n}p_n \quad \mbox{and} \quad \tilde q_n:=\ell^{s_n}q_n,
\end{equation*}
so that $\tilde p_n,\tilde q_n$ are positive integers such that $v(\tilde p_n)=0, v(\tilde q_n)=e_0$ and $p_n/q_n=\tilde p_n/\tilde q_n$. From~\eqref{E.convergents}, we have the following recurrence formulae for $\tilde p_n$ and $\tilde q_n$:

\begin{equation} \label{tilde.recurrence}
\begin{aligned}
    &\begin{cases}
    \tilde p_n=0 &\forall n< 0, \\
    \tilde p_0=1,\quad \tilde p_1=\ell^{e_0}a_0, \\
    \tilde p_{n+1}=\ell^{e_n}(a_n \tilde p_n+\ell^{e_{n-1}} \tilde p_{n-1})  &\forall n\ge 1;
    \end{cases}\\
     &\begin{cases}
    \tilde q_n=0 &\forall n<0, \\
    \tilde q_0=0,  \quad \tilde q_1=\ell^{e_0}, \\
    \tilde q_{n+1}=\ell^{e_n}(a_n \tilde q_n+ \ell^{e_{n-1}} \tilde q_{n-1}) &\forall n\ge 1.
    \end{cases}
\end{aligned}
\end{equation}

The following Lemma, which can be proved by an easy induction using~\eqref{tilde.recurrence}, gives an estimate on the growth of the $\tilde p_n$ and $\tilde q_n$. 

\begin{lem}\label{lemma.lowerbound}
If $a_0\neq 0$, then for every $n>1$, we have
\begin{align*}
    \tilde{p}_{n}&>\ell^{s_n}\geq \ell^{n-1}  &\text{and} \quad  &\tilde{q}_{n}\geq\ell^{s_{n-1}}\geq \ell^{n-2}   &\text{if $n$ is even}\\
     \tilde{p}_{n}&>\ell^{s_{n-1}}\geq \ell^{n-2}  &\text{and} \quad  &\tilde{q}_{n}\geq\ell^{s_{n}}\geq \ell^{n-1}  &\text{if $n$ is odd}.
\end{align*}
\end{lem}


\subsection{On the convergence in \texorpdfstring{$\R$}{R}} \label{SS.r}

Given a RCF with partial quotients $a_0, a_1, \ldots$, we want to analyse the convergence of the continued fraction $[a_0, a_1,\ldots]$ in $\R$ (though not being \textit{the} continued fraction expansion of any real number). 

From~\eqref{E.convergents2} we have that, for every $n\ge 1$, 
   \begin{equation} \label{difference_pq}
   \frac{p_n}{q_n}-\frac{p_{n-1}}{q_{n-1}}=\frac{(-1)^n}{q_n q_{n-1}}.
   \end{equation}
   Since $q_{n+2}\ge q_n$ for every $n\ge 0$ (as follows from the recurrence formulae), this easily  implies  in particular that the sequence $\left (p_n/q_n\right )$ is increasing for odd $n$ and decreasing for even $n$, so these two corresponding subsequences have limit in $\R$ equal to their supremum, for $n$ odd, and their infimum, for $n$ even, respectively. 
   
   Also, the same equation yields that every $p_n/q_n$ with odd $n$ is smaller than any $p_m/q_m$ with even $m\ge 0$  (where we agree that $p_0/q_0=1/0=+\infty$).

We have the following general result:

\begin{prop} \label{P.convergence} A continued fraction with real partial quotients $a_0, a_1,\ldots$ such that $a_i>0$ for $i>0$ converges in $\R$ if and only if $\sum a_i=\infty$, which holds if and only if the $q_n$ are unbounded.
\end{prop}

This fact appears in more general form in  \cite{H}, but we can give a very simple and short proof.

\begin{proof}  For $n\ge 1$, we let for this proof $h_n:=q_nq_{n-1}$. 
The formula~\eqref{difference_pq} implies that there is convergence in $\R$ as soon as $\limsup h_n=+\infty$.  In fact, if this happens, it follows that $p_n/q_n- p_{n-1}/q_{n-1}\rightarrow 0$, and therefore the limit of the odd convergents is equal to the limit of the even ones. 

 The usual recurrence formulae imply that  $h_{n+1}=a_nq_n^2+h_{n-1}$, whence $h_n=a_nq_n^2+a_{n-1}q_{n-1}^2+\ldots +a_2q_2^2+h_1$. As $q_1=1$, $q_2=a_1>0$ and, for every $n\ge 3$, $q_n > q_{n-2}$, we have that for $n\ge 1$ the $q_n$ are bounded from below by a strictly positive number. This yields $h_n\gg \sum_{m\le n}a_m$, thus proving convergence in $\R$ of the continued fraction whenever $\sum a_i$ diverges. Since $h_n\to +\infty$ in this case, the $q_n$ are unbounded. 

For the converse, it suffices to note that $q_n\le (a_{n-1}+1)\cdots (a_1+1)$, as follows from an easy induction.\footnote{See Proposition~\ref{P.qnbound} below for a more accurate estimate.} But then $q_n\le \exp (\sum_{m<n}a_m)$, so the $q_n$ are bounded  if $\sum a_i$ converges. Formula~\eqref{difference_pq} then implies that the continued fraction cannot converge.
\end{proof}

Note that we certainly have convergence in $\R$ if $\underline a:= \inf a_n>0$ (which happens in particular when the $e_n$ are bounded). 

Actually, in this case, the growth of the $q_n$ and the convergence occur with (at least) exponential rate; indeed, it is easy to prove that $p_n,q_n \gg c^n$, where $c$ is the positive root of the equation $x^2- \underline a x-1$; in turn, this root is checked to be $\ge 1+\underline a/2$. 


\subsection{Upper bounds for \texorpdfstring{$\alpha_n,p_n,q_n$}{an,pn,qn}} \label{Bound.pq}
In this short subsection we collect some simple inequalities which do not seem to be easily located in the literature, despite their probable usefulness.

First we give an upper bound for the sequences of the $p_n,q_n$.For this it is not necessary to restrict to the present context, and we consider an arbitrary continued fraction with real positive entries. For $a\in\R$, $a>  0$, we let  

\[ 
B(a):=
\begin{pmatrix}
a&   1\\
1 & 0 
\end{pmatrix}, \qquad \p_n:=\begin{pmatrix}p_n\\p_{n-1}\end{pmatrix},  \qquad  \q_n:=\begin{pmatrix}q_n\\q_{n-1}\end{pmatrix}.
\]
Also, let $\lambda(a)$ denote the maximum eigenvalue of $B(a)$. Of course, $\lambda(a)$ is the positive root of the equation $x^2-ax-1=0$, \textit{i.e.} 
$$
\lambda(a)=\frac{a+\sqrt{a^2+4}}{2},
$$
with the positive square root. We note that the other eigenvalue of $B(a)$ is in $(-1,0)$, whereas $\lambda(a)\geq 1+\frac{a}{2}$.

\begin{prop}\label{P.qnbound} For all $n> 1$ we have
$$
 p_n\le  \lambda(a_{n-1})\cdots \lambda(a_0) \quad \mbox{and} \quad q_n\le  \lambda(a_{n-1})\cdots \lambda(a_1).
$$
\end{prop}

\begin{proof} Let us discuss first the case of the $q_n$.
We have $\q_{m+1}=B(a_m)\q_m$ for all integers $m\ge 1$, whence we derive the well-known matrix representation
$$
\q_{m+1}=B(a_m)\cdots B(a_1)\q_1.
$$
Note that the $B(a)$ are symmetric matrices, and that, since $a\in\R^+$ and $\lambda(a)$ is  the maximum eigenvalue (also in modulus) of $B(a)$, we have for every $\x\in\R^2$ the well-known (easy)  inequality
$$
|B(a)\x|\le \lambda(a)|\x|,
$$
where $|\x|$ denotes the euclidean length. On iterating this and recalling that $q_1=1$ and $q_0=0$, we have:
\begin{equation*}
q_n\le \sqrt{q_n^2+q_{n-1}^2}= |\q_{n}|\le \lambda(a_{n-1})\cdots \lambda(a_1)|\q_1|= \lambda(a_{n-1})\cdots \lambda(a_1),
\end{equation*}
which concludes the argument. 

To prove the estimate for the $p_n$, one argues similarly, using the matrix representation
$$
\p_{m+1}=B(a_m)\cdots B(a_0)\p_0,
$$
and recalling that $p_0=1$ and $p_{-1}=0$ by definition.
\end{proof}

Now we give an upper bound for the height of the complete quotients $\alpha_n$. This holds in the context of Ruban continued fractions.
\begin{prop}\label{prop.bound.altezza.quozienti.completi}
Let $\alpha\in \Q_\ell$ algebraic over $\Q$ with $v(\alpha)\leq 0$. Then for every $n\ge 0$ the following bounds for the height of the complete quotients of the RCF expansion of $\alpha$ hold:
\begin{align}
    \label{height_cq2} h(\alpha_n)&\leq  h(\alpha)+s_n \log \ell + n\log(2\ell), \\
    \label{height_cq} h(\alpha_n)&\leq 2^n(h(\alpha)+\log(2\ell))-\log(2\ell).
\end{align}
\end{prop}
\begin{proof}
By the recurrence formula
\[
\alpha_n=a_n+\frac{1}{\alpha_{n+1}},
\]
we obtain that
\begin{equation}\label{eqn.bound.altezza.quozienti.completi}
h(\alpha_{n+1})=h\left(\frac{1}{\alpha_{n+1}}\right)=h(\alpha_n-a_n)\leq h(\alpha_n)+h(a_n)+\log 2.
\end{equation}
The height of $a_n$ can be easily estimated in terms of $e_n$, because $a_n$ is a positive rational number smaller that $\ell$ and with denominator $\ell^{e_n}$, so that $h(a_n)\leq \log(\ell^{e_n+1}-1)<(e_n+1)\log\ell$. 
The bound~\eqref{height_cq2} now follows by induction from~\eqref{eqn.bound.altezza.quozienti.completi}.

\medskip

On the other hand we have by definition $e_n=-v(\alpha_n)>0$, so that $h(\alpha_n)\geq e_n \log \ell$. Together with~\eqref{eqn.bound.altezza.quozienti.completi} this gives
\[
h(\alpha_{n+1})\leq  h(\alpha_n)+h(a_n)+\log 2\leq h(\alpha_n)+(e_n+1)\log\ell+\log 2\leq 2 h(\alpha_n)+\log(2\ell),
\]
and the second bound follows again by induction on $n$.
\end{proof}


\subsection{Examples}
Let us calculate the RCF expansions of some numbers in $\Q_3$. 
\begin{example}
Take $\alpha=\frac{17}{11}=1+3+3^3+ \cdots$; then, $a_0=1$ and $r_0=\frac{17}{11}-1=\frac{6}{11}$. If we expand $\frac{11}{6}$, we have that $\left \lfloor\frac{11}{6}\right \rfloor_{3}=\frac{1}{3}$ and $r_1=\frac{11}{6}-\frac{1}{3}=\frac{3}{2}$. This means that $\frac{1}{r_1}$ is equal to its integral part, so $a_2=\frac{2}{3}$, $r_2=0$ and the algorithm stops giving the finite expansion
\[
\frac{17}{11}=\left [1, \frac{1}{3}, \frac{2}{3}\right ]_3.
\]
\end{example}
\begin{example}\label{ex2}
Take $\alpha=\frac{5}{6}=3^{-1}+2+2\cdot 3+ \cdots.$ Then, $a_0=\frac{7}{3}$ and $r_0=\frac{5}{6}-\frac{7}{3}=-\frac{3}{2}$.  Going on, we have
\begin{align*}
a_1=\left \lfloor-\frac{2}{3} \right\rfloor_{3}=\frac{7}{3}, \ \ \ &  r_1={-\frac{2}{3}-\frac{7}{3}}=-{3}, \\
a_2=\left \lfloor-\frac{1}{3}\right \rfloor_{3}=\frac{8}{3}, \ \ \ & r_2={-\frac{1}{3}-\frac{8}{3}}=-{3}.
\end{align*}
This means that $-\frac{1}{3}$ has purely periodic continued fraction equal to $\left [\frac{8}{3},\frac{8}{3},\frac{8}{3},... \right ]_3$. The continued fraction expansion of $\frac{5}{6}$ is then equal to 
\[
\frac{5}{6}=\left [\frac{7}{3}, \frac{7}{3},\frac{8}{3},\frac{8}{3},\frac{8}{3},\ldots \right ]_3.
\]
This shows that  even some positive rational numbers may have infinite RCF expansion. 
\end{example}

\begin{example} \label{ex.l}
If we take a prime $\ell$, we have
$
-\frac{1}{\ell}= \sum_{k=-1}^{\infty} (\ell-1) \ell^k
$
  in $\Q_{\ell}$, 
so, if we calculate the RCF expansion, we have that $a_0=\ell-\ell^{-1}$, $r_0=-\ell$; hence, $\frac{1}{r_0}=-\frac{1}{\ell}$ proving that the continued fraction is purely periodic and equal to $\left [\overline{\ell-\ell^{-1}} \right ]_{\ell}$.
Note that this continued fraction converges to $-\ell^{-1}$ in the $\ell$-adic metric, whereas it converges to $\ell$ (the only other possible limit value) in the usual euclidean topology. 

We will see in the next section that, if a rational number doesn't have a terminating RCF expansion, then it has a periodic part equal to $\left [\overline{\ell-\ell^{-1}} \right ]_{\ell}$. 

\end{example}
\begin{example} \label{ex1}
Take $\delta$ the only square root of $37$ in $\Q_3$ congruent to  $1$ modulo $3$. Let us consider $\theta=\frac{1+\delta}{6}$. We have that $\delta=1+2\cdot 3^2+3^4+ \cdots$ for some remainder in $3^5\Z_3$, so $a_0=\frac{1}{3}$ and $\alpha_1=\frac{1}{\frac{\delta-1}{6}}=\frac{\delta+1}{6}=\theta$. This means that the RCF expansion of  $\theta$ is purely periodic and equal to 
$$ \theta=\left [\frac{1}{3},\frac{1}{3},\frac{1}{3},\dotsc \right ]_3.$$

If we consider the RCF expansion of $\delta$, this does not look likely to be periodic. Carrying out the computations we obtain
\[
\delta= \left [ 1,\frac{5}{9},\frac{16}{9},\frac{7}{3},\frac{26}{81}, \frac{5}{3}, \frac{7}{9}, \frac{7}{3}, \dotsc \right ]_3.
\]
  We will show later that this expansion is indeed not periodic.
\end{example}


\section{Rationals and the terminating case: Theorem~\ref{T.terminating}}  \label{S.terminating} 

Of course if a number in $\Q_\ell$ has a terminating RCF expansion, then it is rational and positive. Hence a negative rational number cannot have a finite continued fraction. However from Example~\ref{ex2} it appears that even positive rational numbers may have infinite continued fraction expansions. Indeed, the RCF cannot terminate if in the expansion  we find a negative complete quotient. This proves the easy implication in part (i) of Theorem~\ref{T.terminating}. We now prove the rest of the statement.
\subsection{Proof of part (i) of Theorem~\ref{T.terminating}}
Simultaneously with the proof,  we give an algorithm to test whether the expansion of a given (positive) rational $\alpha\in\Q$ is terminating. The algorithm works by computing sufficiently many complete quotients and checking whether they are positive. 

\medskip

We can start by computing the first $m+1$ partial quotients, complete quotients and convergents, assuming $m=2k\ge 2$. If any of the complete quotients computed so far is negative, then the algorithm stops, and we can conclude that the RCF expansion does not terminate. Then, we may assume that $\alpha_0,\alpha_1,\ldots ,\alpha_{2k}\ge 0$. 

As showed in Section~\ref{SS.r}, the sequence of convergents $p_n/q_n$ is increasing for odd $n$ and decreasing for even $n$, and formula~\eqref{E.cq2} with $2k$ in place of $n$ shows that  $\alpha$ lies in between $p_{2k-1}/q_{2k-1}$ and $p_{2k}/q_{2k}$. By equation~\eqref{E.approx}, we have:
\begin{equation} \label{phi_2k}
0\le \varphi_{2k}=p_{2k}-\alpha q_{2k}= \frac{1}{\alpha_{2k}q_{2k}+q_{2k-1}}\le \frac{1}{q_{2k-1}},
\end{equation}
 as we are assuming $\alpha_0,\alpha_1,\ldots ,\alpha_{2k}\ge 0$.\\
On the other hand, if $b=\ell^{e_0}b_0>0$ is a denominator for $\alpha$,  where $b_0$ is a positive integer not divisible by $\ell$, the number $b_0(p_{2k}-\alpha q_{2k})$ is an integer divisible by $\ell^{s_{2k+1}-e_0}$ as $v(\varphi_{2k})=v(q_{2k+1})$ from~\eqref{E.approx2}. 

Finally, $q_{2k-1}\ge q_1=1$, so if $\ell^{s_{2k+1}}>\ell^{e_0}b_0=b$, the above equation~\eqref{phi_2k} forces $\alpha=p_{2k}/q_{2k}$.
Hence, to decide about this dichotomy we need merely to perform the algorithm until $\ell^{2k+1}>b$.

\medskip

This proves the first part of Theorem~\ref{T.terminating}; moreover, we deduce the following: 

\medskip

\noindent{\bf Quantified algorithm (i):} {\it If the Ruban continued fraction expansion of a rational number $\alpha$ with denominator $b$  is not terminating, then a negative complete quotient will appear in at most  $\max\left(\frac{\log b}{\log\ell},2\right)$ steps.}
\medskip

Using the same arguments, we can also prove the following conclusion:

\begin{prop}\label{T.ratconv} 
The Ruban continued fraction expansion of a rational number always converges in $\R$.
\end{prop}

\begin{proof}  We have seen in Section~\ref{SS.r} that, if a RCF expansion does not converge in $\R$, then the $q_n$ are bounded, so  if $\alpha\in\Q$, the numbers $| p_n-\alpha q_n|$  are rationals with bounded denominators; but since $p_n/q_n$ are bounded, the numerators are  also bounded. Using~\eqref{E.approx2},  we see that they are divisible by $\ell^{s_{n+1}-e_0}$. This is eventually impossible unless they vanish, but then the RCF expansion of $\alpha$ is finite, proving anyway our conclusion. 
\end{proof}

Indeed, this conclusion follows also directly from part (ii) of Theorem~\ref{T.terminating}, which we prove in next section.


\subsection{Proof of part (ii) of Theorem~\ref{T.terminating}} 
We want now to prove that, if the RCF expansion of a rational number $\alpha$ does not terminate, then it is periodic with all the partial quotients eventually equal to $\ell-\ell^{-1}$. Moreover, we will give an explicit bound for the length of the pre-periodic part.

\medskip

We assume that $v(\alpha)< 0$, which can be achieved by replacing $\alpha$ with $\alpha_1$, and we write $\alpha=\alpha_0=\frac{d}{b\ell^{e_0} }$, where $b,d$ are coprime integers not divisible by $\ell$, $b>0$ and $e_0> 0$ consistently with our notation. Assume also that the RCF expansion of $\alpha$ does not terminate. The idea of the proof is to consider again the quantities $b\varphi_n=b(p_n-\alpha q_n)$. Arguing as in the previous section, this number is an integer, because thanks to the factor $b$ the denominator can only be a power of $\ell$, and we know from~\eqref{E.approx2}  that $v(p_n-\alpha q_n)=s_{n+1}-e_0\geq n$.
Therefore we can write $b\varphi_n=\beta_n \ell^{s_{n+1}-e_0}$ for some integers $\beta_n$, which are not zero because the continued fraction does not terminate.

Concerning the usual absolute value, by Proposition~\ref{P.qnbound}, we have
\begin{equation*}
\abs{p_n-\alpha q_n}\leq (\lambda(a_0)+|\alpha|)\lambda(a_{n-1})\cdots \lambda(a_1),
\end{equation*}
where $\lambda(a)=\frac{a+\sqrt{a^2+4}}{2}$ as seen before. We have then
\begin{equation}\label{1.estimate}
1\leq \abs{\beta_n}\leq  b(\lambda(a_0)+|\alpha|)\ell^{e_0-s_{n+1}}\prod_{i=1}^{n-1}\lambda(a_i)\leq b(\lambda(a_0)+|\alpha|)\ell^{-e_n}\prod_{i=1}^{n-1}\frac{\lambda(a_i)}{\ell^{e_i}}.
\end{equation}
As every partial quotient $a$ is of the form $\frac{r}{\ell^{e}}$, where $e\ge 1$ is an integer and $1\le r \le \ell^{e+1}-1 $, it is easy to prove that
$$
\lambda(a)\le \ell^e,
$$
with the equality holding if and only if $a=\ell-\ell^{-1}$.

We also have that $b(\lambda(a_0)+|\alpha|)\le b(\ell^{e_0}+\abs{\alpha}) \leq 2H(\alpha)$, so that from~\eqref{1.estimate} we obtain for $n>0$
\begin{equation}\label{eqn:bound.beta}
\abs{\beta_n}\leq  \frac{2 }{\ell^{e_n}}H(\alpha) \leq \frac{2}{\ell} H(\alpha),
\end{equation}
independently of $n$. This shows that the $\abs{\beta_n}$ belong to a fixed finite set of cardinality at most $\frac{2}{\ell} H(\alpha)$ (remember that we are assuming $e_0>0$, so $H(\alpha)\geq\ell$).

Moreover, as $\lambda(a)$ is an increasing function of $a\ge 0$, we have 
 \begin{enumerate}
\item if $e\ge 2$, then
$\lambda(a)\ell^{-e}\le \lambda(\ell)\ell^{-e}\le \frac{(1+\sqrt{2})\ell}{2\ell^{e}}\le \left (\frac{1+\sqrt{2}}{2\ell} \right )^{e-1} \le\left (\frac{1+\sqrt 2}{2\ell}\right )^{\frac{e}{2}}\le \left (\frac{1+\sqrt 2}{4}\right )^{\frac{e}{2}};$
\item if $e=1$ and $a\neq \ell-\ell^{-1}$, then $\lambda(a)\ell^{-1}\le \lambda \left (\ell-\frac{2}{\ell} \right)\ell^{-1}=\frac{\ell^2-2+\sqrt{\ell^4+4}}{2\ell^2} \le  1-\frac{3}{4\ell^2}$.
\end{enumerate}
But for every $\ell \ge 2$ we have that $ \left (\frac{1+\sqrt 2}{4}\right )^{\frac{1}{2}}< 1-\frac{3}{4\ell^2}$, so $\lambda(a)\ell^{-e} \le \craz^e $ where $\craz:=1-\frac{3}{4\ell^2}<1$, unless $a=\ell-\ell^{-1}$.

Using this last estimate in~\eqref{1.estimate} and putting $\sigma_n:=\sideset{}{'}\sum e_i$, the sum being extended to all indices $1\le i\le n-1$ with $a_i\neq \ell-\ell^{-1}$, we have
\[
1\leq b(\lambda(a_0)+\abs{\alpha})\ell^{-1}\craz^{\sigma_n}\leq \frac{2}{\ell} H(\alpha)\craz^{\sigma_n}.
\]
Hence $\sigma_n$ is bounded independently of $n$. 
 This shows that only finitely many $a_n$ can be different from $\ell-\ell^{-1}$, so the continued fraction expansion is periodic and all the partial quotients are eventually equal to this number. 

To conclude the proof, we are going to exhibit an explicit bound for the length of the pre-periodic part in the case that the continued fraction expansion does not terminate.

Consider the identity $\alpha_n=-\frac{\varphi_{n-1}}{\varphi_n}=-\frac{\beta_{n-1}}{\ell^{e_n}\beta_n}$. We have from~\eqref{eqn:bound.beta} that all the $\alpha_n$ belong to a fixed finite set of cardinality at most $\frac{8}{\ell} H(\alpha)^2$ and have height bounded by $2H(\alpha)$. Then, for some $i<j\leq \frac{8}{\ell} H(\alpha)^2 +1$, we will find $\alpha_i=\alpha_j$ and the continued fraction becomes periodic from $\alpha_i$ on, with all partial quotients equal to $\ell-\ell^{-1}$ as proved before.

This computation holds under the assumption $e_0>0$. It might be needed to replace $\alpha_0$ with $\alpha_1$. In this case the length of the preperiodic part is increased by one, while the height $H(\alpha)$ is replaced by $H(\alpha_1)\leq 2\ell H(\alpha)$ (by Proposition~\ref{prop.bound.altezza.quozienti.completi} if $v(\alpha)=0$, otherwise $\alpha_1=1/\alpha$ and they have the same height).

This completes the proof of Theorem~\ref{T.terminating}; moreover, we deduce the following:

\medskip

\noindent{\bf Quantified algorithm (ii):} {\it If the Ruban continued fraction expansion of a rational number $\alpha$ is not terminating, then the length of the pre-periodic part can be explicitly computed as above; in particular it is at most $32\ell H(\alpha)^2$.

Moreover for all $n\geq 0$ the height of the complete quotients is bounded by
\[
H(\alpha_n)\leq 4\ell H(\alpha).
\]
}
\ 
\begin{rem} \label{remark.rational}
We point out that one can compare the complexity of a rational number $\alpha=\frac{p}{q}$ with a finite RCF expansion with the length $k$ of the expansion itself. In fact, it is easy to prove using Lemma~\ref{lemma.lowerbound}, that
\begin{equation*} \label{tilde.estimates}
k \le \frac{\log \min\{p,q\}}{\log\ell}+2\leq \frac{h(\alpha)}{\log\ell}+2.
\end{equation*}
\end{rem}
\medskip


\subsection{Finiteness of the expansion for varying \texorpdfstring{$\ell$}{l}}

Given $\alpha \in \Q\subseteq \Q_\ell$, we can ask what happens to the expansion when we vary the prime $\ell$.
The following proposition gives an answer to this question:

\begin{prop}
Let $\alpha\in\Q$. The following holds:
\begin{enumerate}
\item[(i)] If $\alpha<0$, then for every prime number $\ell$ the RCF expansion of $\alpha$ does not terminate;
\item[(ii)] If $\alpha\geq 0$ and $\alpha\in\Z$, then there are only finitely many prime numbers $\ell$ such that the RCF expansion of $\alpha$ does not terminate;
\item[(iii)] If $\alpha\geq 0$ and $\alpha\not\in\Z$, then there are only finitely many prime numbers $\ell$ such that the RCF expansion of $\alpha$ terminates.
\end{enumerate}
\end{prop}
\begin{proof}
 Part (i) follows directly from Theorem~\ref{T.terminating}. 
 
 Assume that $\alpha \in \Z_{\ge 0}$ has finite RCF expansion. Then, by Remark~\ref{remark.rational}, the length of this expansion is at most $2$. More specifically, if $\alpha < \ell$, then $\lfloor \alpha \rfloor_{\ell}=\alpha$ and so the RCF expansion is equal to $[\alpha]_\ell$ with length one. This proves the assertion~(ii). We point out that we can say something more about the behaviour of the RCF expansion in $\Q_{\ell}$ with $\ell \le \alpha$. In this case in fact, $\lfloor \alpha \rfloor_{\ell} \neq \alpha$ so the expansion has necessarily length $2$, and this happens if and only if $\alpha=\ell^h+a_0$ with $0\le a_0 < \ell$ and $h\ge 1$. 
 
 Let now take $\alpha=\frac{n}{m}$ with $n,m$ positive coprime integers and $m>1$. We prove that, if $\ell$ is a prime number with $\ell>\max\{n,m\}$, then the RCF expansion of $\alpha$ in $\Q_{\ell}$ does not terminate. 
 
 Let $a_0=\pfloor{\alpha}$; we have that $a_0\in\Z$ because $\ell\nmid m$.
 If $\frac{n}{m}-a_0\geq 0$, then we would have that $\ell>n>n-ma_0> 0$, which is impossible because $n-ma_0$ is divisible by $\ell$. Then we see that $\alpha_1=\frac{1}{\frac{n}{m}-a_0}$, which is the first complete quotient of the RCF expansion of $\alpha$, is negative. Then by Theorem~\ref{T.terminating} the expansion does not terminate. This completes the proof of the proposition.
\end{proof}


\section{About quadratic expansions}\label{S.special} 

In this section we analyse the behaviour of the RCF expansion of the quadratic irrationals $\alpha\in\Q_\ell$. For such $\alpha$, we denote by $\alpha'$ the algebraic conjugate over $\Q$.
 
From a quadratic equation $A\alpha^2+B\alpha+C=0$, with $A,B,C\in\Z$ and $A\neq 0$, we derive the shape 
   \begin{equation} \label{E.shape}
   \alpha=\frac{b_0+\delta}{\ell^{f_0} c_0},
   \end{equation}
where $\delta\in\Q_\ell$, $\delta^2=\Delta$ is a non-square integer (non necessarily square-free) and $b_0,c_0,f_0$ are integers with $\ell \nmid c_0$. 

A necessary and sufficient condition for $\delta$ to lie in $\Q_{\ell}$ can be obtained with Hensel's lemma. Let us write $\Delta=\ell^{h}\tilde{\Delta}$ with $\ell\nmid\tilde{\Delta}$; then $\delta\in\Q_{\ell}$ if an only if $h$ is even and $\tilde{\Delta}$ is a square modulo~$\ell$, if $\ell$ is odd, or $\tilde{\Delta}\equiv 1 \pmod 8$, if $\ell=2$. In particular $\ell$ will always appear with even exponent in the values of $\Delta$ that we consider.

Notice that the $b_0,\delta,f_0,c_0$ in~\eqref{E.shape} are not uniquely determined. We could obtain uniqueness by imposing a coprimality condition, but for our aims it is more convenient to allow multiple representations and common factors between the numerator and the denominator. We can always require
\begin{equation} \label{first.shape}
c_0 \mid \Delta - b_0^2.
\end{equation}
In fact this is not restrictive since it can be achieved for instance by replacing $b_0,c_0, \Delta$ respectively by $b_0c_0, c_0^2, c_0^2\Delta$.

Recall that $e_0=-v(\alpha)$; then, $e_0=f_0-v(b_0+\delta)$, so $e_0 \le f_0$. Under the non-restrictive assumption $\vl(\alpha)\leq 0$, we have that $f_0\geq 0$.

\medskip

We want now to detail the behaviour of the first step in the $\ell$-adic continued fraction expansion of $\alpha$.
\medskip

Take $a_0=\lfloor \alpha \rfloor_{\ell}=\frac{r_0}{\ell^{e_0}}$, where $r_0$ is the unique integer with  $0\le r_0\le \ell^{e_0+1}-1$ such that $b_0+\delta-c_0 r_0 \ell^{f_0-e_0}\equiv 0 \pmod {\ell^{f_0+1}}$, where this congruence holds in $\Z_{\ell}$. 
In this case, $r_0\neq 0$. We can then write $\alpha$ as:
   $$
   \alpha= \frac{r_0}{\ell^{e_0}}+\frac{\delta+b_0-c_0r_0\ell^{f_0-e_0}}{\ell^{f_0} c_0} =\frac{r_0}{\ell^{e_0}}+\frac{\Delta-(c_0r_0\ell^{f_0-e_0}-b_0)^2}{\ell^{f_0} c_0(\delta+(c_0r_0\ell^{f_0-e_0}-b_0))}.
   $$
Now, if we denote $b_1:= c_0r_0\ell^{f_0-e_0}-b_0$, we easily find that  
$\Delta-b_1^2\equiv \Delta-b_0^2\equiv 0 \pmod {c_0}$ and $\ell^{f_0+1} \mid [\delta-(c_0r_0\ell^{f_0-e_0}-b_0)] \mid \Delta-b_1^2$. This means that we can write $\Delta-b_1^2=\ell^{f_0+f_1}c_0 c_1$ with $f_1\ge 1$ and $(\ell, c_1)=1$. Hence, we have
$$
   \alpha=a_0+\frac{1}{\alpha_1},  
$$
with 
$$
a_0:=\frac{r_0}{\ell^{e_0}}, \qquad \mbox{and} \qquad \alpha_1:=\frac{b_1+\delta}{\ell^{f_1}c_1},
$$
where $r_0, f_1, b_1,c_1$ are defined as above. 

In particular, this shows that the first complete quotient  $\alpha_1$ (and hence all the subsequent ones) has the same shape of $\alpha$, especially in the fact that $\delta$ in the numerator appears with the coefficient $1$. This is due to our special hypothesis that $ c_0\ |\ \Delta-b_0^2$ (which, as noted above, can always be achieved and is preserved at every step).

\medskip

We can continue our expansion and, as a matter of notation, we shall put
 \begin{equation*}
   \alpha_m=\frac{b_m+\delta}{\ell^{f_m}c_m}=[a_m,a_{m+1},\ldots]_\ell,
   \end{equation*} 
where $\ b_m,\ c_m$ and $f_m$ are integers defined by the following recurrence formulae:
\begin{equation} \label{quadratic.recurrence} 
\begin{cases}
b_n+b_{n+1} &=a_n\ell^{f_n}c_n \\
\ell^{f_n+f_{n+1}}c_nc_{n+1} &=  \Delta - b_{n+1}^2.
\end{cases}
\quad \mbox{for all } n\ge 0.
\end{equation}

Notice that, since we are assuming that $\ell$ does not divide any $c_n$,  the second formula of~\eqref{quadratic.recurrence} implies that $f_n+f_{n+1}=\vl(\Delta-b_{n+1}^2)$; hence, as $f_n\ge 0$, we have for every $n\ge 0$,
\begin{equation*}
1\le f_{n+1}\le v(\Delta-b_{n+1}^2),
\end{equation*} 
where the second inequality becomes strict when $n\ge 1$ as $f_n\ge 1$.

\bigskip

We just saw that, computing the RCF expansion of a quadratic $\alpha$ satisfying~\eqref{first.shape}, all complete quotients will be quadratic numbers in $\Q_{\ell}(\alpha)$ of a similar shape. Moreover once we fix $(\Delta,b_0,c_0,f_0)$ the recurrence~\eqref{quadratic.recurrence} defines sequences $b_n,f_n,c_n$ uniquely. We will now show that, after a finite number of steps depending only on $\alpha$, we reach a complete quotient $\alpha_M$ which admits a (possibly different) representation $(\tilde\Delta,\tilde{b}_M,\tilde{c}_M,\tilde{f}_M)$ which satisfies some additional conditions. More specifically, we have the following proposition:
\begin{prop} \label{Prop.Part.shape}\mbox{}
\begin{enumerate} 
\item[(i)] Let $\alpha=\frac{b_0+\delta}{\ell^{f_0} c_0}$ as before satisfying~\eqref{first.shape}. Assume in addition that 
\begin{equation} 
\begin{cases} \label{Part.shape}
\ell^{f_m}c_m \mid \Delta-b_m^2, \\
\ell \nmid {\Delta}, \\ 
v(\alpha_m)=-f_m<0 \quad &\mbox{if $\ell$ is odd}, \\
v(\alpha_m)=1-f_m<0 \quad &\mbox{if $\ell=2$}
\end{cases}
\end{equation}
holds for $m=0$. Then~\eqref{Part.shape} holds for every $m\geq 0$ as well.

\item[(ii)] Assume that $\Delta=\ell^{2h}\tilde\Delta$, with $h\ge 0$ and $(\ell, \tilde\Delta)=1$. 
Then, there exists a positive integer $M \le h+2$ such that
$\alpha_{M}=\frac{\tilde{b}_{M}+\tilde{\delta}}{\ell^{\tilde{f}_{M}} \tilde{c}_{M}}$ with ${\tilde{\delta}}^2=\tilde{\Delta}$ and $\tilde{b}_{M}, \tilde{c}_{M}, \tilde{f}_{M}$ integers satisfying~\eqref{Part.shape}.
\end{enumerate}
\end{prop}

\begin{proof}
We prove first the part (i) of the statement.

The first two conditions in~\eqref{Part.shape} are clearly preserved by the recurrence formulae~\eqref{quadratic.recurrence}. We only have to show that the condition on $\vl(\alpha_m)$ is preserved as well.

Suppose first that $\ell$ is odd. Then $\vl(b_1-\delta)\geq f_0+1\geq 2$ by construction, while $\vl(2\delta)=0$. Therefore by the ultrametric inequality we have $\vl(b_1+\delta)=0$ and $\vl(\alpha_1)=-f_1$.

If $\ell=2$ instead, we have $\vl(b_1-\delta)\geq f_0+1\geq 2$ by construction, while $\vl(2\delta)=1$. Therefore by the ultrametric inequality we have $\vl(b_1+\delta)=1$ and $\vl(\alpha_1)=1-f_1$.

\bigskip

We now prove part (ii). After changing the representation of $\alpha$ and replacing $(b_0,\Delta,f_0)$ with $(b_0/\ell^k,\Delta/\ell^{2k},f_0-k)$ if needed, we may suppose that $\ell\nmid (b_0,\Delta).$
Let us write now $\Delta= \ell^{2h}\tilde{\Delta}$ for some integer $h\geq 0$.

Let us first assume $h>0$. As by assumption $\ell \nmid (\Delta, b)$ and $\ell \mid \Delta$, we have that $\ell \nmid b_0+\delta$, hence $v(b_0+\delta)=0$ and $e_0=-v(\alpha)=f_0$.

Then by the algorithm discussed at the beginning of the section, we have  
\[ \alpha_1=\frac{b_1+\delta}{\ell^{f_1} c_1},\]
where $b_1=c_0r_0-b_0$ as shown before. But, by construction, we have that $\delta-b_1 \equiv 0$ $\pmod {\ell^{f_0+1}}$ and $\ell^h \mid \delta$, hence $b_1\equiv 0 \pmod{ \ell^{\min\{h, f_0+1\}} }$. This means that we can simplify the factor $\ell^{\min\{h, f_0+1\}}$ (whose exponent is $\ge 1$), obtaining 
$\alpha_1=\frac{\tilde{b}_1+\delta_1}{\ell^{\tilde{f}_1}c_1}$, with $v(\delta_1)\le h-1$ and $0>\vl(\alpha_1)=\vl(\tilde{b}_1+\delta_1)-\tilde{f}_1$, so that $\tilde{f}_1> 0$. If $\vl(\delta_1)>0$, then we can repeat the argument. In this way we see in at most $h$ steps we reach $\alpha_m=\frac{\tilde{b}_m+\tilde{\delta}}{\ell^{\tilde{f}_m}c_m}$ with $\tilde{\Delta}=\tilde{\delta}^2$ satisfying $(\ell, \tilde{\Delta})=1$ and $\tilde{f}_m> 0$. Notice that, in the last step, we are no more sure that $v(\tilde{b}_m+\tilde{\delta})=0$ as $(\ell, \tilde{\Delta})=1$, so we could in principle have $ \tilde{f}_m \neq -\vl(\alpha_m)$.

We have shown that, up to performing the previous procedure and replacing $\Delta$ with $\tilde{\Delta}$, we can assume that $\ell\nmid\Delta$.

\bigskip

Let us now analyse what happens if $f_0 \neq -\vl(\alpha)$. 
We will show that, after possibly replacing $\alpha$ with its second complete quotient, we can always satisfy the assumption $f_0=-\vl(\alpha)>0$ if $\ell$ is odd or $f_0=1-\vl(\alpha)$ if $\ell=2$.

\bigskip

Take $\alpha=\frac{b_0+\delta}{\ell^{f_0} c_0}$. As shown before, we can assume that $(\ell, \Delta)=1$. In general, we have that $-e_0=\vl(\alpha)=\vl(b_0+\delta)-f_0$, so $e_0\le f_0$.

Then, performing the first step of the algorithm as explained at the beginning of the section, we obtain $\alpha_1=\frac{\delta+b_1}{\ell^{f_1}c_1}$ with $\ell^{f_1}c_1 \mid \Delta- b_1^2$.

Let us distinguish two cases: 

\begin{itemize}
   \item if $\ell$ odd, then by assumption $\ell \nmid \Delta$ so that $\vl(2\delta)=0$. By construction we have that $\vl(\delta-b_1)\ge 1$ and so by the ultrametric inequality $\vl(\delta+b_1)=0$ and $e_1=-\vl(\alpha_1)=f_1$ as wanted.
                 
   \item assume now that $\ell=2$. If $e_0=f_0\ge 0$, this means that $\vl(\delta+b_0)=0$, hence as $2 \nmid \Delta$, $b_0$ is even. Now, performing the first step and using the first equation of~\eqref{quadratic.recurrence}, we have that $b_1$ is odd, hence $\vl(\delta+b_1)\ge 1$. So, after one step, we reduced ourselves to the case $e_1>f_1> 0$.  \\
Assume now $e_0 > f_0$, \textit{i.e.} $\vl(\delta+ b_0)>0$. By~\eqref{quadratic.recurrence}, we have that 
\begin{equation} \label{quadratic.recurrence2}
   f_0+f_1= \vl(\Delta-b_1^2)=\vl(\delta+b_1)+\vl(\delta-b_1). 
\end{equation}
Now, we have by construction that both $\vl(\delta+b_1)$ and $\vl(\delta-b_1)$ are $\ge 1$ and moreover $\min\{\vl(\delta+b_1),\ \vl(\delta-b_1) \}=1$.  Using in~\eqref{quadratic.recurrence2} that $f_1-e_1=\vl(\delta+b_1)$, we have that $\vl(\delta-b_1)=f_0+e_1 \ge 2$, as by assumption $f_0\ge 1$ and $e_1=-\vl(\alpha_1)\ge 1$. But this shows that $\vl(\delta+b_1)=1$, hence $\vl(\alpha_1)=1-f_1 <0$, as wanted. The same holds for all subsequent steps, proving the statement.\qedhere
\end{itemize}
\end{proof}

\begin{definition} \label{D.shape} We shall refer to $\Delta=\delta^2$  in~\eqref{E.shape} as the \textit{ordinate} of $\alpha$ for this shape. Note that this is not uniquely determined by $\alpha$ (but it is determined up to a square).
\end{definition}

\medskip

\begin{example}
Take $\ell=3$ and let $\delta$ be the only square root of $13$ in $\Q_3$ which is congruent to $1$ modulo $3$. We want to compute the $\ell$-adic expansion of $\delta$.
\begin{itemize}
   \item $a_0=\lfloor\delta\rfloor_3=1$, so $\alpha_1=\frac{1}{\delta-1}=\frac{\delta +1}{12}$;
   \item $a_1=\lfloor\alpha_1\rfloor_3=\frac{2}{3}$, so $\alpha_2=\frac{\delta - 7}{12}$.
\end{itemize}
Going on with the calculations, we have that
$\delta= \left[1, \frac{2}{3}, \frac{4}{3}, \frac{8}{9}, \dotsc \right]_3$, which does not seem to have a periodic pattern. 
\end{example}


\subsection{Useful bounds}\label{Bound.lineare.em}
Using the recurrence relations given by~\eqref{quadratic.recurrence} it is possible to give exponential bounds for the quantities $b_m$ and $c_m\ell^{e_m}$.

\begin{prop}\label{P.bound.em}
The following bounds hold:
\begin{align*}
\abs{b_n}&\ll_{\alpha}\ceigen(\ell)^n, \\
\abs{c_n\ell^{f_n}}&\ll_{\alpha}\ceigen(\ell)^n,\ \\
f_n&<3n+O_{\alpha}(1),
\end{align*}
where $\ceigen(\ell)=\frac{\ell^2+2+\ell\sqrt{\ell^2+4}}{2}$ and the implied constants are effectively computed below.
\end{prop}
\begin{proof}
For this proof, we write $k_n:=\ell^{f_n}c_n\in\Z$ for all $n\ge 0$. Moreover, we define $k_{-1}:=\frac{\Delta-b_0^2}{\ell^{f_0}c_0}$. By the definition of $b_{n+1}$, we can write $b_{n+1}=a_n k_n-b_n$; then~\eqref{quadratic.recurrence} gives, for every $n\ge 0$,
\[
k_{n+1}=\frac{\Delta - {b_{n+1}}^2}{k_n}=\frac{\Delta - {b_{n}}^2}{k_n}+\frac{a_nk_n(2b_n-a_nk_n)}{k_n}=k_{n-1}+a_n(2b_n-a_nk_n).
\]
Now $0\leq a_n<\ell$ by construction, so that
\[
\abs{k_{n+1}}<\abs{k_{n-1}}+2\ell \abs{b_n}+\ell^2\abs{k_n},
\]
and
\[
\abs{b_{n+1}}<\abs{b_n}+\ell\abs{k_n}.
\]
If we define two recurrence sequences $A_n,B_n$ such that $A_{-1}=\abs{k_{-1}},$ $A_0=\abs{k_0},$  $B_0=\abs{b_0}$, and
\begin{align*}
\begin{cases}
A_{n+1}=\ell^2A_n+A_{n-1}+2\ell B_n, \\
B_{n+1}=\ell A_n+B_n,
\end{cases}
\end{align*}
then we have the estimates $\abs{k_n}\leq A_n$ and $\abs{b_n}\leq B_n$ for all $n\ge 0$.

Now we can argue as in Section~\ref{Bound.pq} and write 
\[
M=
\begin{pmatrix}
\ell^2&   1 & 2\ell\\
1 & 0 &0\\
\ell & 0 & 1
\end{pmatrix}, \qquad v_n:=\begin{pmatrix}A_n\\A_{n-1}\\B_n\end{pmatrix}, \qquad v_{n+1}=Mv_n.
\]
The biggest eigenvalue of $M$ is the number $\ceigen(\ell):=\frac{\ell^2+2+\ell\sqrt{\ell^2+4}}{2}< \ell^2+2$, and therefore
\[
A_n,B_n\leq| {v_n} |=|M^n v_0| \le \ceigen(\ell)^n |v_0|,
\]
where $|v_0|\le |k_{-1}|+|k_0|+|b_0|=:\ceigensnd(\alpha)$.
In particular, we have that
\begin{align*}
\abs{b_n}&\leq B_n\leq \ceigensnd(\alpha) \ceigen(\ell)^n,\\
\abs{c_n\ell^{f_n}}&\leq A_n\leq \ceigensnd(\alpha) \ceigen(\ell)^n,\\
f_n&\leq \frac{\log A_n}{\log\ell}\leq n \frac{\log \ceigen(\ell)}{\log \ell}+\frac{\log\ceigensnd(\alpha)}{\log \ell}<\frac{\log(3+2\sqrt{2})}{\log 2}n+\frac{\log\ceigensnd(\alpha)}{\log \ell} \\
&<3n+\ceigentrd(\alpha),
\end{align*}
where $\ceigentrd(\alpha)=\frac{\log\ceigensnd(\alpha)}{\log \ell}$.
\end{proof}


\subsection{Convergence in \texorpdfstring{$\R$}{R} of the Ruban continued fraction expansion of quadratic irrationals} We have briefly discussed convergence in $\R$ of a general RCF, and we have also proved directly in Proposition~\ref{T.ratconv} that for Ruban expansions of rational numbers, convergence in $\R$ always holds.
In this section we prove the same for the expansion of an arbitrary quadratic irrational number. 

\begin{prop}\label{T.quadconv} 
The Ruban continued fraction expansion of a quadratic irrational number always converges in $\R$.
\end{prop}

\begin{proof} We may assume the number $\alpha$  in question is of the form~\eqref{E.shape} 
and use the results of the previous section. As in the previous proposition, for this proof we write $k_n:=\ell^{f_n}c_n\in\Z$; then by~\eqref{quadratic.recurrence} we have 
\begin{equation} \label{eq.Rec}
    b_{n+1}+b_n=a_nk_n \quad \text{and} \quad \Delta-b_{n+1}^2=k_nk_{n+1}.
\end{equation} 

Assume by contradiction that the RCF expansion of $\alpha$ does not converge in the real topology. Then we have that $a_n\ge \ell^{-f_n}$ and we deduce from Proposition~\ref{P.convergence} that $f_n\to +\infty$ as $n\to +\infty$; moreover, since $k_n=\ell^{f_n} c_n$ and the $c_n$ are non-zero integers, this means that $|k_n|\rightarrow +\infty$, hence there exists $n_0 \ge 1$ such that $|k_n|> \Delta$ for all $n\ge n_0$. Since we may replace $\alpha$ with any partial quotient, we may assume that this happens for all $n\ge 0$. Formulae~\eqref{eq.Rec} above show that $k_n k_{n+1}<0$ and  $b_{n+1}^2=\Delta+|k_n||k_{n+1}|$. By shifting $n$ if necessary, we may assume that $k_n>0$ for even $n$.

From~\eqref{eq.Rec} for $n$ and $n+1$ in place of $n$, we have $b_{n+2}-b_n=a_{n+1}k_{n+1}-a_nk_n$, so the sequence of the $b_n$ is monotone increasing for odd  $n$ and monotone decreasing for even $n$; also, we find that, for instance for odd $n$,  $b_{n+2}=a_{n+1}|k_{n+1}|+a_n|k_n|+\ldots +a_1|k_1|+b_1$ and similarly for even $n$. We also find that eventually $b_n$ is positive (resp. negative) for odd (resp. even) $n$, and on shifting again we may assume this holds for all $n$. We further find 
\begin{equation} \label{eq1.R}
|b_{n+1}|-|b_n|=a_n|k_n|.
\end{equation}

Equation~\eqref{eq.Rec} also yields 
\begin{equation} \label{eq2.R}
(|b_{n+1}|-|b_n|)(|b_{n+1}|+|b_n|)=|k_n|(|k_{n+1}|-|k_{n-1}|).
\end{equation}

Set for this proof $\gamma_n:=\sqrt{|k_n|}$. From the equality  $b_{n+1}^2= |k_n| |k_{n+1}|+|\Delta|$, we derive  $|b_{n+1}|\le \gamma_n\gamma_{n+1}+\sd$, whence 
\begin{equation*}
|b_n|+|b_{n+1}|\le \gamma_n(\gamma_{n+1}+\gamma_{n-1})+2\sqrt{|\Delta|}.
\end{equation*}
Using this in~\eqref{eq2.R}, we find
\begin{equation*}
|k_n|(\abs{k_{n+1}}-|k_{n-1}|)\le 
(|b_{n+1}|-|b_n|)\left( \gamma_n(\gamma_{n+1}+\gamma_{n-1}) +2\sqrt{|\Delta|}\right).
\end{equation*}
On using~\eqref{eq1.R} and dividing by $|k_n|(\gamma_{n+1}+\gamma_{n-1})$, we have
\begin{equation*}
\gamma_{n+1}-\gamma_{n-1}\le a_n \left ( \gamma_n + \frac{2 \sd}{\gamma_{n+1}+ \gamma_n} \right ),
\end{equation*}
hence, as $\gamma_{n}\ge 1$ for all $n\ge 1$,
\begin{equation*}
\gamma_{n+1}\le a_n \left ( 1 + \sd\right )\gamma_n+\gamma_{n-1}.
\end{equation*}
Now, $\sum a_n$ converges by assumption; so, arguing as in the previous proposition, it easily  follows that $\gamma_n$ are bounded, a contradiction which proves the result. 
\end{proof}

This result allows to formulate the following (probably  difficult) \medskip

{\bf Conjecture}: {\it Consider the real limit  of the Ruban continued fraction expansion of a quadratic irrational. Then, either the expansion is eventually periodic  or the said limit  is transcendental.}

\medskip

Other problems concern the behaviour of the $e_n$ for a non periodic RCF expansion of a quadratic irrational. Can they tend to infinity? It will follow from the arguments at the end of the paper that they cannot be bounded.


\section{Purely periodic expansions}\label{S.purely}

  Before studying periodicity, let us focus on {\it pure periodicity}, \textit{i.e.} the case when an $\alpha\in\Q_\ell$ has a RCF expansion which is purely periodic. 
  
   If this happens, we have $\alpha=\alpha_k$ for some period $k>0$, and equation~\eqref{E.cq} implies that $\alpha$ satisfies a quadratic equation over $\Q$, precisely, 
\begin{equation}\label{E.quadratic}
q_k\alpha^2-(p_k-q_{k-1})\alpha-p_{k-1}=0.
\end{equation}

This equation holds for every period $k$, and so for any multiple of any given period. Using this, we can prove Theorem~\ref{T.periodic}.

\begin{proof} [Proof of Theorem~\ref{T.periodic}]  In view of Theorem~\ref{T.terminating}, we obtain that  the only rational number with a purely periodic RCF expansion is    $\alpha=\ell-\ell^{-1}$.
   
In all other cases of pure periodicity, equation~\eqref{E.quadratic} says that $\alpha$ is quadratic irrational. But in any case $\Q(\alpha)$ can be embedded in $\R$, since $k\geq 1$ and so $q_k,p_{k-1}>0$. 
   
Now, if $\alpha\in\Q_\ell$ has a(n eventually) periodic RCF expansion, then  some complete quotient has a purely periodic expansion, and the above conclusion applies. This proves Theorem~\ref{T.periodic}.
\end{proof}

Notice that, if $\alpha\in\Q_\ell\setminus\Q$ has a pure periodic continued fraction, it is then a quadratic irrational and by Proposition~\ref{Prop.Part.shape} we have that $\alpha$ has directly to satisfy the condition given by~\eqref{Part.shape}.
We can derive the following useful proposition:

\begin{prop}\label{P.ppp}
If $\alpha\in\Q_\ell\setminus\Q$ has a purely periodic Ruban continued fraction expansion, then there is precisely one embedding $j=j_\alpha:\Q(\alpha)\to\R$ such that  $j(\alpha_m)>0$ for all complete quotients $\alpha_m$.

Similarly, but relative to the $\ell$-adic valuation, we have $|\alpha|_\ell>1$ and  $|\alpha'|_\ell=|\alpha|_\ell^{-1}< 1$.
\end{prop}

\begin{proof}
To prove the first assertion, let $\alpha$ be as in the statement. From equation~\eqref{E.quadratic}, we have that there is precisely one real embedding $j:\Q(\alpha)\to\R$ such that $\xi:=j(\alpha)>0$ (in fact, the product $\xi \xi'=-p_{k-1}/q_k$ is negative) and we may repeat the argument for each complete quotient. We are left with the task of showing that $j$ is the same for all complete quotients, and, by induction, it suffices to show that $j$  works for the first complete quotient $\alpha_1$ as well, \textit{i.e.} that $\xi_1:=j(\alpha_1)>0$. 

We have $\alpha=a_0+\alpha_1^{-1}$, so $\xi=a_0+\xi_1^{-1}$. Taking conjugates, we have $\xi'=a_0+\xi_1'^{-1}$. Now, $\xi'<0$ so $\xi_1'=(\xi'-a_0)^{-1}<0$, and hence $\xi_1>0$, as wanted.

To prove the last assertion concerning the $\ell$-adic place, notice first that $|\alpha|_\ell>1$ because, by periodicity, $\alpha$ is equal to some other complete quotient, and every complete quotient after the first one has this property. Moreover,
equation~\eqref{E.quadratic} shows that $|\alpha\alpha'|_\ell=|p_{k-1}/q_k|_\ell$,
hence 
$$ |\alpha'|_\ell= |p_{k-1}/q_k|_\ell/ |\alpha|_\ell= \ell^{e_0+s_{k-1}-s_k-e_0}=\ell^{-e_{k-1}},$$
and the assertion follows since $e_{k-1}\geq 1$.
\end{proof}

\begin{rem}[\texttt{Dependence of the embedding}]\label{R.emb}  It is worth noting that, even if the quadratic field $K=\Q(\alpha)\subset \Q_\ell$ is given, in general the embedding $j_\alpha$ of the previous proposition is not uniquely determined by $K$, namely it may change if we take another quadratic $\beta\in K$ with a purely periodic RCF expansion. Here is an example of this behaviour.
   
Let $\delta\in\Q_7$ be the square root of $2$ such that $\delta\equiv 3\pmod 7$ and put $\alpha=\frac{1+5\delta}{7}$. Then  $\alpha$ is a root of $x^2-\frac{2}{7}x-1$, so $\alpha=\frac{2}{7}+\alpha^{-1}$. By our choice of $\delta$, we have $|\alpha|_7>1$, hence the previous equation produces the purely periodic expansion $\alpha=\left [\overline{\frac{2}{7}}\right ]_7$.  If we let $\sqrt 2$ denote the real positive square root of $2$, the embedding $\iota:\Q(\alpha)=\Q(\delta)\to\R$ such that $\iota(\delta)=\sqrt 2$ has $\iota(\alpha)>0$, hence is the (unique) embedding $j_\alpha$ predicted by the proposition for $\alpha$.
   
Let now $\beta:=\frac{17-13\delta}{7}$, so $\Q(\beta)=\Q(\alpha)$. This $\beta$ is a zero of $x^2-\frac{34}{7}x-1$, and since again $|\beta|_7>1$ (note that  $17-13\cdot 3\equiv -1\pmod 7$) we have $\beta= \left [\overline{\frac{34}{7}} \right ]_7$. On the other hand, $\iota(\beta)=\frac{17-13\sqrt 2}{7}<0$.  Hence the embedding $j_\beta$ of the proposition relative to $\beta$ in place of $\alpha$ is $\iota_{-}$, where with $\iota_-$ we denote the other real embedding of $\Q(\beta)$ into $\R$, \textit{i.e.} the one such that $\iota_-(\delta)=-\sqrt 2$. 
   \end{rem}
   
\begin{example}[\texttt{Purely periodic RCF of period 1 in a quadratic field}]\ \\
Fix $\ell$ a prime number and let $\delta \in \Q_{\ell}\setminus \Q$ be the square root of a positive integer $\Delta$.
We want to compute the pure periodic RCF's of $\Q(\delta)$, \textit{i.e.} RCF of the form $\left [ \overline{\frac{t}{\ell^h}} \right ]_{\ell}$ with $h>0$ and $0\le t < \ell^{h+1}$. If $x=\left [\overline{\frac{t}{\ell^h}} \right]_{\ell}$, this means that $x$ is the solution of the equation $x^2-\frac{t}{\ell^h}x-1=0$ with $|x|_{\ell}>1$, \textit{i.e.} $x$ is of the form $\frac{t\pm\sqrt{t^2+4\ell^{2h}}}{2\ell^{h}}$. In order to have that $x\in \Q(\delta)$, we must impose that $t^2+4\ell^{2h}=u^2\Delta$ for some $u \in \Z$. Hence, to generate all pure periodic RCF's of period 1 of $\Q(\delta)$, we have to solve the generalized Pell equations 
   \[
   t^2-u^2\Delta= -4\ell^{2h}, \quad \mbox{with $0\le t<\ell^{h+1}$}.
   \]
   We give an explicit example.  Take $\ell=3$ and let $\delta\in \Q_3$ be the square root of $\Delta=10$ which is congruent to $1$ modulo $3$. Let us calculate the first solutions of the previous equation:
\begin{itemize}
     \item for $h=1$, we have $t=1$ and $u=1$, hence $x_1=\frac{1+\delta}{3}=\left [\overline{\frac{2}{3}}    \right ]_3$;
     \item for $h=2$, we have $t=13$ and $u=5$, hence $x_2=\frac{13-5\delta}{9}=\left [ \overline{\frac{26}{9}} \right]_3$;
     \item for $h=3$, we have $t=31$ and $u=13$, hence $x_3=\frac{31+13\delta}{27}=\left [\overline{\frac{62}{27}} \right ]_3$;
     \item for $h=4$, we have $t=43$ and $u=29$, hence $x_4=\frac{43-29\delta}{81}=\left [\overline{\frac{86}{81}} \right ]_3$;
     \item for $h=5$, no solutions;
     \item for $h=6$, we have two solutions, giving $x_5=\frac{881+289\delta}{3^6}=\left [ \overline{\frac{1762}{3^6}} \right]_3$ and $x_6=\frac{601-205\delta}{3^6}=\left [ \overline{\frac{1202}{3^6}} \right ]_3$.
   \end{itemize}
   
   Now, let $\sqrt{10}$ be the real positive square root of 10, and denote by $\iota$ the embedding $\iota:\Q(\delta)\to\R$ such that $\iota(\delta)=\sqrt{10}$. We easily see that, in the previous examples, $\iota(x_1), \iota(x_3)$ and $\iota(x_5)>0$, while $\iota_-(x_2), \iota_-(x_4)$ and $\iota_-(x_6)>0$ (where $\iota_-$ denotes the other real embedding of $\Q(\delta)$ in $\R$). This gives another evidence that the embedding of the proposition does not depend only on the field $\Q(\delta)$.
   
\end{example}


\subsection{Finiteness of purely periodic expansions with a given shape}\label{SS.shape}
   
We want now to study the quadratic irrationals with a given ordinate which have purely periodic continued fractions.
   
First of all, as observed before, if $\alpha \in \Q_\ell \setminus \Q$ has a purely periodic RCF expansion, it is a quadratic irrational satisfying conditions~\eqref{Part.shape}. So, let us consider $\alpha$ of the form
   \begin{equation}\label{E.shape2}
      \frac{b_0+\delta}{\ell^{f_0} c_0},\qquad \ell \nmid \delta^2=\Delta\in\Z_{>0},\quad \ell \nmid  \, c_0\in\Z,\quad f_0\in\Z_{>0},
   \end{equation}
where $\Delta>0$ since by Proposition~\ref{P.ppp} $\Q(\alpha)$ embeds in $\R$.
Moreover, we have that $\vl(\alpha)=-f_0<0$ if $\ell$ is odd and $\vl(\alpha)=1-f_0<0$ if $\ell=2$.
\medskip
   
We know that, by Section~\ref{S.special}, every complete quotient of the RCF expansion of $\alpha$ has the form 
\[
 \alpha_m=\frac{b_m+\delta}{\ell^{f_m}c_m},
\] 
with $b_m, c_m, f_m$ satisfying the conditions~\eqref{Part.shape}. Denote by $\xi_m$ and $\xi_m'$ the two real embeddings of $\alpha_m$. Using Proposition~\ref{P.ppp}, we have that $\xi_m \xi_m' <0$ for all $m \ge 0$. But if we compute $\xi_m\xi_m'$, we have that 
\[
\xi_m\xi_m'= \frac{b_m^2 - \Delta}{\ell^{2f_m}c_0^2} <0,
\]
which implies that $b_m^2 < \Delta$ for all $m \ge 0$. In particular, we proved that $|b_m|< \sqrt{\Delta}$, so the sequence $|b_m|$ is bounded.
   
Now, by~\eqref{quadratic.recurrence} and the fact that $|b_m|$ is bounded, we have that 
\begin{equation} \label{E.estimate}
     \ell^{f_m+f_{m+1}} |c_m c_{m+1}| \le \Delta.  
\end{equation}
   
   In turn, for given $\Delta$, this implies in particular that $f_m$  and $|c_m|$ are bounded, so they have a finite number of possibilities  in terms of $\Delta$. As the $|b_m|$ are bounded too, we get a finite list of numbers of the form~\eqref{E.shape2}, with a given ordinate, which satisfy the necessary condition \eqref{E.estimate} to be purely periodic.
   
\medskip
  
  Once we have this finite list {\it containing} those $\alpha$ satisfying \eqref{E.estimate}, we can effectively determine the precise list of the purely periodic ones. For this, given $\beta$ in the list, it suffices to compute more complete quotients of $\beta$  than the cardinality of the list: either we find some complete quotient out of the list (and then we rule out $\beta$ itself) or we must find a repetition, which would provide the full period. We state all of this in a proposition:
  
\begin{prop} \label{P.shape}
For a given non-square integer $\Delta>0$, there are only finitely many $\alpha\in\Q_\ell$ with ordinate $\Delta$ such that the RCF expansion of $\alpha$ is purely periodic, and these numbers may be effectively determined.
\end{prop}

\begin{rem}[\texttt{Practical computations and an example}]\label{R.practical}
For practical computations, once we form a list according to the above inequalities, we may often shorten it by applying Proposition~\ref{P.ppp}, possibly also looking at the first complete quotient (on looking both at the real and at the $\ell$-adic valuation): if any conjugate has either positive real value or $\ell$-adic absolute value $\ge 1$, we may eliminate that number from the list. 
  
Let us give here an example, taking $\ell=3$ and $\delta$ the square root of $\Delta=13$ in $\Q_3$  which is congruent to $1$ modulo $3$ in $\Q_3$. Let us compute the list of the $\alpha\in\Q_3$ with ordinate $13$ (\textit{i.e.} of the form $\frac{b_0+\delta}{3^{e_0} c_0}$) which have a purely periodic RCF expansion. 
  
 Using~\eqref{E.estimate}, we have, for all $m\ge 1$, 
$3^{f_m+f_{m-1}} |c_mc_{m-1}| \le 13$. 
Since $f_h\ge 1$ for all $h$, this immediately implies that, for all $m\ge 0$, $f_m=1$ and $c_m=\pm 1$; also, $|b_m| < \sqrt{13}$. Hence the possible elements of the previous shape having a purely periodic continued fractions are among the following fourteen ones: $ \pm\frac{\delta}{3}, \   \pm\frac{\pm 1+\delta}{3},\  \pm\frac{\pm 2+\delta}{3},\  \pm\frac{\pm 3+\delta}{3}$.

Now, by Proposition~\ref{P.ppp} we must have $|\alpha|_{3}>1$ and  $|\alpha'|_{3}<1$; hence we easily see that the only two {\it a priori} possible cases for a purely periodic continued fraction with ordinate $13$ are $ \pm\frac{-2+\delta}{3}$.  We now find that the first complete quotient of $\frac{-2+\delta}{3}$ is $-\frac{7+\delta}{12}$, which does not belong to the list. Then  the minus sign must occur, so we remain with $\alpha=\frac{2-\delta}{3}$. 
    This satisfies $\alpha^2-\frac{4}{3}\alpha-1=0$, so $\alpha=\frac{4}{3}+\frac{1}{\alpha}$, leading indeed to the purely periodic RCF expansion $\alpha=\left [\overline{\frac{4}{3}} \right ]_3$.
\end{rem}


\section{A general criterion for periodicity}\label{S.periodic} 

In this section, we are going to prove a necessary and sufficient criterion to decide whether a quadratic irrational $\alpha$ has a periodic RCF expansion or not.

\subsection{An explicit example}
We begin by showing a concrete example involving the square root of $13$ in $\Q_3$.

\begin{prop}\label{T.eta} Let us denote by $\delta\in\Q_3$ the square root of $13$ which is congruent to $1$ modulo $3$. Then, the RCF expansion of $\delta$ is not periodic.
\end{prop}
            
\begin{proof}  Suppose by contradiction that the RCF expansion of $\delta$ is periodic; then, there will be some complete quotient of the expansion which is purely periodic. As discussed at the beginning of Section~\ref{S.special}, every complete quotient in the RCF expansion of $\delta$ will have ordinate $13$, and we proved in Remark~\ref{R.practical} that the only quadratic irrational number with ordinate $13$ which has purely periodic RCF expansion is $\frac{2-\delta}{3}=\left [\overline{\frac{4}{3}} \right ]_3$. This means that there exist $n>0$ and $a_0, \ldots, a_{n-1}\in \Z\left [ \frac{1}{3}\right ]$ such that $\delta=\left [a_0, \ldots, a_{n-1}, \overline{\frac{4}{3}} \right ]_3$. From~\eqref{E.cq}, we have that 
\begin{equation*}
   \delta= \frac{\alpha p_n+ p_{n-1}}{\alpha q_n+ q_{n-1}},
\end{equation*}
with $\alpha=\frac{2-\delta}{3}$, hence, on substituting for $\alpha$, we get 
\begin{equation*}
   \delta=\frac{-p_n\delta+3p_{n-1}+2p_n}{-q_n\delta+3q_{n-1}+2q_n}.
\end{equation*}
If we now denote by $\xi$ the positive real square root of $13$ and by $j:\Q(\delta) \rightarrow \R$ the real embedding sending $\delta$ to $-\xi$, we have that
\begin{equation*}
    -\xi=j(\delta)=\frac{p_n\xi+3p_{n-1}+2p_n}{q_n\xi+3q_{n-1}+2q_n},
\end{equation*}
which gives a contradiction as the right member of the equality is positive while $-\xi<0$. 
\end{proof}

\begin{rem}\label{R.-delta} This method naturally works much more generally. For instance, let us see that $-\delta$ does not have a periodic $3$-adic continued fraction either. The argument of Proposition~\ref{T.eta} does not apply directly, but we can compute a few more complete quotients and try to repeat the argument with those numbers.

We have
\[
   -\delta=2+(-2-\delta)=2+\frac{-9}{-2+\delta},
\]
hence the first complete quotient is $(2-\delta)/9$.
The second complete quotient is found to be equal to $(-11+\delta)/12$. Now, if $-\delta$ has a RCF expansion which is periodic, the same would hold for all complete quotients, in particular for $(-11+\delta)/12$. But now we can perform the same argument as in Proposition~\ref{T.eta} (replacing $\delta$ with $(-11+\delta)/12$) and we obtain a contradiction.
\end{rem}

\begin{rem}
The proofs of Proposition~\ref{T.eta} and of Remark~\ref{R.-delta} ultimately rely on the independence of the real and $3$-adic topologies.
  
It is also possible to give another argument for the non-periodicity of the RCF expansion of $-\delta$ which avoids any reference to the convergence in the reals. The idea is to examine the RCF at the point where the periodic part begins and to seek a contradiction using a congruence argument.     
   
Repeating the argument used in Proposition~\ref{T.eta} for the RCF expansion of $-\delta$, we find that 
\begin{equation}\label{eta3}
  -\delta=\frac{-p_n\delta+3p_{n-1}+2p_n}{-q_n\delta+3q_{n-1}+2q_n}
\end{equation}
for some $n>0$.
We may also assume that $n$ is the minimal integer for which this holds. Comparing rational and irrational parts, this equation yields
\begin{equation}\label{E.congr}
3p_{n-1}+2p_n=13q_n,\qquad 3q_{n-1}+2q_n=p_n,
\end{equation}
and, eliminating $p_n$, we have that
\[
   3q_n=2q_{n-1}+p_{n-1}.
\]
From this equation we can clear the denominators and obtain $ 3\tilde q_n=3^{e_{n-1}}(2\tilde q_{n-1}+\tilde p_{n-1}),$ which implies that $e_{n-1}=1$ and that 
\[
  \tilde q_n=2\tilde q_{n-1}+\tilde p_{n-1},
\] 
while clearing denominators from~\eqref{E.congr} we have that $2\tilde q_n\equiv \tilde p_n\pmod 9$, and clearing denominators from $p_nq_{n-1}-p_{n-1}q_n=(-1)^{n}$ yields $\tilde q_n\tilde p_{n-1}\equiv \tilde p_n\tilde q_{n-1}\pmod 9$.
  
Hence  $\tilde q_n\tilde p_{n-1}\equiv 2\tilde q_n\tilde q_{n-1}$, whence $ 2\tilde q_{n-1}\equiv \tilde p_{n-1}\pmod 9.$
This implies that $\tilde q_n\equiv 4\tilde q_{n-1}\pmod 9$, which easily gives $a_{n-1}=4/3$. But this implies $\alpha_{n-1}=\left [\overline{\frac{4}{3}} \right ]_3$, which contradicts our choice of the index $n$ as the smallest for which~\eqref{eta3} holds.
\end{rem}


\subsection{The general result}

In this section we give a general effective criterion to decide whether the RCF expansion of a quadratic irrational is periodic or not.

Let us consider a quadratic irrational $\alpha \in \Q_\ell \setminus \Q$. We can assume by Theorem~\ref{T.periodic} that $\Q(\alpha)$ can be embedded in $\R$. As argued before we can take it of the form 
\begin{equation*}
\alpha= \frac{b_0+ \delta}{\ell^{f_0}c_0}, 
\end{equation*}
with $b_0, c_0, \delta^2=\Delta>0$ and $f_0$ satisfying conditions~\eqref{first.shape}. As discussed before, all complete quotients will be of the form $\alpha_m=\frac{b_m+\delta}{\ell^{f_m}c_m}$, where $b_m,\ c_m$ and $f_m$ are integers defined by the recurrence formulae~\eqref{quadratic.recurrence}.

We saw in a previous example how to deduce that a certain RCF is not periodic, by looking at the sign of a suitable complete quotient and considering its real embeddings. The sign of the real embeddings is related to the size of the quantities $b_m$ (more precisely, to the sign of $\Delta-b_m^2$). We now show that the periodicity of the RCF expansion is related to the boundedness of the $b_m$.

\begin{prop}
   Let $\alpha\in \Q_{\ell}$ be quadratic over $\Q$. Then, its Ruban continued fraction expansion is periodic if and only if the sequence $\{\abs{b_n}\}_{n}$ is bounded from above.
\end{prop}

\begin{proof}
Assume first that $\alpha$ has a periodic RCF expansion. Then, there exist $M_0\ge 0$ and $k>0$ such that for every $m\ge M_0$, $\alpha_{m+k}=\alpha_m$. This means that we have a finite number of complete quotients $\{\alpha, \alpha_1, \ldots, \alpha_{M_0+k-1}\}$. But every $\alpha_n$ has the form $\frac{b_n+\delta}{\ell^{f_n}c_n}$ for all $n=0, \ldots, M_0+k-1$, and the $b_n$ can assume only finitely many values, hence the sequence $\{|b_n|\}_n$ is bounded as required.

Conversely, assume that the sequence $\{|b_n|\}_n$ is bounded; as they are integers, they can assume only finitely many values. Moreover, from~\eqref{quadratic.recurrence}, we have that, for all $n\ge 0$, 
$$ \ell^{f_n+f_{n+1}} c_n c_{n+1} = \Delta- b_{n+1}^2. $$
The $c_n$ are non-zero integers and, by Proposition~\ref{Prop.Part.shape}, after a finite number of steps, also the $f_n$ are all positive; therefore, for all $n\ge n_0$, we have $|\ell^{f_n}c_n|\le \abs{\Delta- b_{n+1}^2} $. This implies that there exists a finite number of possibilities for the $f_n$ and the $c_n$. But now as  $\alpha_n = \frac{b_n+\delta}{\ell^{f_n}c_n}$, all complete quotients of $\alpha$ vary among a finite number of possibilities. Hence there exist some $N,M\ge 0$ such that $\alpha_N=\alpha_M$, which implies that the RCF expansion of $\alpha$ is periodic.
\end{proof}

The previous proposition gives a necessary and sufficient condition for deciding whether a RCF expansion is periodic or not, but it does not give an effective criterion. In the following theorem we give an effective bound on the number of complete quotients we have to compute in order to decide on periodicity.

\begin{thm}\label{general_condition}
Let $\alpha\in\Q_\ell$ be a quadratic irrational over $\Q$. 

If the Ruban continued fraction expansion of $\alpha$ is eventually periodic, then there is a unique  embedding $j:\Q(\alpha)\to\R$ such that $j(\alpha_n)>0$ for every complete quotient $\alpha_n$ of $\alpha$; moreover, the  limit  in $\R$  of the same continued fraction is  a root of the same quadratic equation satisfied by $\alpha$, and in fact is $j(\alpha)$. 

Conversely, there is a computable integer $N=N_\alpha>0$ such that, if for every $n\le N$ there is an embedding $j_n:\Q(\alpha)\to \R$ with $j_n(\alpha_n)>0$, then the Ruban continued fraction expansion of $\alpha$ is periodic.
\end{thm}

Before proving Theorem~\ref{general_condition}, we show a simple lemma.
\begin{lem}\label{remark.negativi.restano.negativi}
Let $\alpha\in\Q_\ell$ be a quadratic irrational over $\Q$ and $j:\Q(\alpha)\to\R$ an embedding such that $j(\alpha_m)<0$ for some complete quotient $\alpha_m$. Then, $j(\alpha_r)<0$ holds for all $r\geq m$.

In particular, if the embeddings $j_n$ exist as in the second part of Theorem~\ref{general_condition}, then the same embedding can be chosen independently of $n$.
\end{lem}
\begin{proof}
By definition, for every $m\ge 0$ we have that $\alpha_{m}=a_m+\alpha_{m+1}^{-1}$. Therefore $j(\alpha_{m})=a_m+j(\alpha_{m+1})^{-1}$. If $j(\alpha_m)<0$, then $j(\alpha_{m+1})$ must be negative as well, because $a_m>0$. By induction this proves the first part of the statement.

For the second part we just need to remark that, if for a fixed embedding $j$ and complete quotient $\alpha_m$, we have $j(\alpha_m)<0$, then the same embedding $j$ will take negative values over all complete quotients from $\alpha_m$ on. We can then take all the embeddings $j_n$ equal to the other embedding $j'\neq j$.
\end{proof}

\begin{proof}
 [Proof of Theorem~\ref{general_condition}] Suppose first that $\alpha$ is a quadratic number in $\Q_\ell\setminus \Q$ having a(n eventually) periodic RCF expansion. As discussed in the previous proposition, for every $m$ greater or equal than some $m_0$, all complete quotients $\alpha_m$ of $\alpha$ have a purely periodic RCF expansion. 
Therefore, by Proposition~\ref{P.ppp}, there exists a unique embedding $j:\Q(\alpha)\to \R$ such that $j(\alpha_m)>0$ for all $m \ge m_0$. In particular, $\Q(\alpha)$ is a real quadratic field. 

Now, let us use again the fact that, for every fixed integer $h\ge 0$, $\alpha_{m+h}$ is a complete quotient of $\alpha_h$ for every $m\ge 0$. In particular, if we call $\{r_n/s_n\}_n$ the sequence of convergents to $\alpha_h$, by~\eqref{E.cq}, we have  
\[ 
\alpha_h=\frac{\alpha_{m_0+h}r_{m_0}+r_{m_0-1}}{ \alpha_{m_0+h}s_{m_0}+s_{m_0-1}},
\]  
hence 
\[
j(\alpha_h)=\frac{j(\alpha_{m_0+h})r_{m_0}+r_{m_0-1}}{j(\alpha_{m_0+h})s_{m_0}+s_{m_0-1}}.
\] 
As $j(\alpha_{m_0+h})$ for every $h> 0$ this shows that also $j(\alpha_h)>0$ as required.

\medskip

As seen in Subsection~\ref{SS.r}, $\liminf a_n>0$, which is ensured by periodicity, implies that the RCF expansion of $\alpha$ converges in $\R$. In view of the usual formulae relating $\alpha$ with the complete quotients, the corresponding limit must satisfy the same quadratic equation satisfied by $\alpha$, and the same holds for all complete quotients $\alpha_n$. Also, all of these limits must be  positive, so must be equal to $j(\alpha_n)$. This proves the first part of the statement.

\bigskip

To go in the opposite direction is more delicate. 
Let
\[
\alpha=\frac{b_0+\delta}{\ell^{f_0}c_0}
\]
be a representation of $\alpha$ satisfying the usual conditions~\eqref{E.shape} and~\eqref{first.shape}. We denote by $\xi$ and $\xi'$ the two real embeddings of $\alpha$ and by $\xi_n$ and $\xi_n'$ the real embeddings of $\alpha_n$ for all $n\ge 0$. Let also $t:=\floor{\sqrt{\Delta}}$ (notice that $\Delta>0$ since by assumption $\Q(\alpha)$ can be embedded in $\R$).
\medskip

{\bf Step A)} We show first that within at most $\max(0,\abs{b_0}-t )$ steps in the continued fraction expansion either we reach a complete quotient with two negative embeddings, or a complete quotient $\alpha_n$ satisfying the inequality $\abs{b_n}<\sqrt{\Delta}$.

Suppose that $\abs{b_0}>\sqrt{\Delta}$. As $\xi\xi'=\frac{b_0^2-\Delta}{(\ell^{f_0}c_0)^2}>0$, it follows that the two embeddings have the same sign, which we take to be positive otherwise there is nothing to prove. This in turn implies that $b_0$ and $c_0$ have the same sign.
If the same holds for $\alpha_1$ as well, \textit{i.e.} $\abs{b_1}>\sqrt{\Delta}$ and $\xi_1,\xi_1'>0$, we have similarly that $b_1c_1>0$.

Suppose for example that $b_0>0$. Then $c_0>0$ and the second equation in~\eqref{quadratic.recurrence} with $n=0$ implies that $c_1<0$ and then $b_1<0$. But then the first equation in~\eqref{quadratic.recurrence} gives that $\abs{b_0}-\abs{b_1}>0$. The same holds analogously if $b_0<0$. This shows that the absolute values of the $b_i$ decrease until $\abs{b_i}<\sqrt{\Delta}$ or both $\xi_i<0$ and $\xi_i'<0$ hold; this will happen in at most $\abs{b_0}-t$ steps.
\medskip

{\bf Step B)} By replacing $\alpha$ with a complete quotient, we can now assume that $\abs{b_0}<\sqrt{\Delta}$ holds. Arguing as before, this implies that $\xi\xi'<0$. Assume without loss of generality that $\xi>0$ and $\xi'<0$. By Lemma~\ref{remark.negativi.restano.negativi}, we have that $\xi_n'<0$ for all $n\geq 0$.
If, for a later index $m$, it holds that $\abs{b_m}>\sqrt{\Delta}$, then $\xi_m\xi_m'>0$ which implies $\xi_m,\xi_m'<0$. This shows that the inequality 
$\abs{b_n}<\sqrt{\Delta}$ holds until we reach a complete quotient with two negative embeddings.

Notice now that, under the inequality $\abs{b_n}<\sqrt{\Delta}$, equation~\eqref{quadratic.recurrence} implies that the $c_i$'s all have the same sign.
Therefore for every fixed value of $b_n$, the second equation in~\eqref{quadratic.recurrence} implies that the quantity $\ell^{f_n}c_n$ can assume at most $\Delta-b_n^2$ different values. 
However $b_n$ can assume only the $2t+1$ different values between $-t$ ant $t$.
Thus we see that, after at most $1+\sum_{i=-t}^t(\Delta-i^2)=(2t+1)\Delta-\frac{t(t+1)(2t+1)}{3}+1$ steps, either we reach a complete quotient with two negative embeddings or we get a repetition in the sequence of the complete quotients, which implies the periodicity of the RCF expansion of $\alpha$.
\medskip

Hence, if we take $N_{\alpha}:=\max(0,\abs{b_0}-t )+(2t+1)\Delta-\frac{t(t+1)(2t+1)}{3}+1$, and for all $n\le N_{\alpha}$ at least one between $\xi_n$ and $\xi_n'$ is positive, then the RCF expansion of $\alpha$ is periodic, proving the claim.
\end{proof}

\begin{rem}
We point out that there is an alternative argument to prove Theorem~\ref{general_condition} which uses well-known explicit bounds for the number of solutions of some $S$-unit equation. It was in fact the use of $S$-units which allowed us to find the algorithm using the real embedding. On the other hand, the bound on $N_{\alpha}$ produced with this approach is far worse than the one above. 
\end{rem}

\begin{rem}[\texttt{Explicit computation of the constants}]\label{explicit_constant}
  As shown in the proof, given $\alpha=\frac{b_0+\delta}{c_0\ell^{f_0}}$ of the usual shape~\eqref{E.shape} satisfying~\eqref{first.shape} and setting $t=\floor{\sqrt{\Delta}}$, the value $N_\alpha$ can be taken equal to
\[
 \max(0,\abs{b_0}-t)+(2t+1)\Delta-\frac{t(t+1)(2t+1)}{3}+1.
\]
However, this bound can be improved in the case that $\ell\mid \Delta$ by performing the steps described in Proposition~\ref{Prop.Part.shape} and choosing a more convenient representation for the complete quotients involved. 
  Let $\Delta=\ell^{2h}\tilde{\Delta}$ and $\tilde t=\floor{\sqrt{\tilde{\Delta}}}$. Notice that the sign of the quantities $b_i^2-\Delta_i$ which occur during the cancellation steps of Proposition~\ref{Prop.Part.shape} is not changed when divided by a power of $\ell$. Then if $\abs{b_i}>\sqrt{\Delta_i}$, the absolute value of the $b_i$ decreases at each step unless we find a complete quotient with two negative embeddings. This means that, after  $\max(h+2,\abs{b_0}-t)$ steps, both the cancellations of Proposition~\ref{Prop.Part.shape} and Step A of the proof will be completed. We can then perform Step B using $\tilde\Delta$ instead of $\Delta$, which saves a factor of $\ell^{3h}$ in the main term at the cost of performing at most $h+2\approx\log\Delta$ additional steps.

The value $N_\alpha$ can then be taken equal to
\[
\max(h+2,\abs{b_0}-t)+(2\tilde t+1)\tilde{\Delta}-\frac{\tilde t(\tilde t+1)(2\tilde t+1)}{3}+1.
\]
As $\sqrt{\tilde{\Delta}}-1<\tilde t < \sqrt{\tilde{\Delta}}$, this is bounded from above by
\[
\max\left(h+2,\abs{b_0}-\floor{\sqrt{\Delta}}\right)+\frac{4}{3}\tilde{\Delta}^{3/2}-\frac{1}{3}\tilde{\Delta}^{1/2}+1.
\]
\end{rem}
  
\begin{example}
Let us give some examples of application of Theorem~\ref{general_condition}. \\
Let us prove that the $3$-adic expansion of $\delta$, the only square root of $37$ in $\Q_3$ which is congruent to $1$ modulo $3$ considered in the Example~\ref{ex1} is not periodic. Indeed:
\begin{itemize}
\item $a_0=\lfloor\delta\rfloor_3=1$, hence $\delta= 1+ \frac{1}{\alpha_1}$, with $\alpha_1=\frac{1+\delta}{36}$. If we denote by  $\iota:\Q(\delta)\rightarrow \R$ the real embedding that sends $\delta$ in $\sqrt{37}$, and with $\iota_-$ the one sending $\delta$ in $-\sqrt{37}$, we have that $\iota(\alpha_1)>0$ and $\iota_-(\alpha_1)<0$;
\item $a_1=\lfloor\alpha_1\rfloor_3=\frac{5}{9}$, hence $\alpha_1=\frac{5}{9}+\frac{1}{\alpha_2}$,with $\alpha_2=-\frac{19+\delta}{9}$. So, we have that $\iota(\alpha_2)<0$ and $\iota_-(\alpha_2)<0$. By Theorem~\ref{general_condition}, we can conclude that the expansion of $\delta$ is not periodic.
\end{itemize}
In this case, the algorithm stops at the second step. There are also cases in which more steps are needed to decide whether the expansion is periodic or not. Take for example $\gamma$ the square root of 13 in $\Q_3$ which is congruent to $1$ modulo $3$. Take $\theta=\frac{2+\gamma}{12}$. Then
\[ \theta=\left [0, \frac{2}{3}, \frac{5}{9}, \frac{2}{3}, \frac{8}{3}, \frac{7}{3},\ldots \right ]. \]
In this case, the images of the complete quotients under the two real embeddings of $\Q(\gamma)$ are both negative for the first time at the fifth iteration. Indeed, it is easy to construct examples in which this phenomenon happens arbitrarily late.
\end{example}

\noindent We notice that Theorem~\ref{general_condition} allows us to conclude that certain classes of square roots of positive integers cannot have a periodic continued fraction. For instance, the following holds:
\begin{cor}
Let $\ell$ be an odd prime and $\Delta=1+k \ell^h$ not a square, with $h,k$ positive integers and $(k,\ell)=1$. Let $\delta \in \Q_{\ell} \setminus \Q$ be the square root of $\Delta$ congruent to $1$ modulo $\ell$. Then, if $\Delta > (\ell^h+1)^2$,  the RCF expansion of $\delta$ is not periodic.
\end{cor}

\begin{proof}
Let us compute the RCF expansion of $\delta$:
$a_0=1$, hence $\delta=1+\frac{1}{\alpha_1}$, with $\alpha_1=\frac{1+\delta}{\Delta-1}=\frac{1+\delta}{k\ell^h}$.
If we denote by $\iota:\Q(\delta) \rightarrow \R$ the real embedding obtained sending $\delta$ in $\sqrt{\Delta}$ (the positive real square root of $\Delta$) and by $\iota_-$ the other real embedding, then $\iota(\alpha_1)>0$ and $\iota_-(\alpha_1)<0$.

To calculate the second complete quotient, we have first to compute the $\ell$-integral part of $\alpha_1$. As by assumption $\Delta=1+k \ell^h$, using the Taylor expansion of $\sqrt{1+x}$, we can write $\delta$ as $\delta=1+\frac{k\ell^h}{2}+\ell^{2h}C$ for some $C\in \Z_{\ell}$. This implies that
\[
\lfloor\alpha_1\rfloor_{\ell}=\left \lfloor \frac{\delta+1}{k \ell^h} \right \rfloor_{\ell}= \left \lfloor \frac{2+\frac{k\ell^h}{2}}{k \ell^h} \right \rfloor_{\ell}= \left \lfloor \frac{2}{k \ell^h}+ \frac{1}{2} \right \rfloor_{\ell}.
\]
Now, as $v\left(\frac{1}{2} \right)= 0$, we have that $\lfloor\alpha_1 \rfloor_{\ell}\ge \frac{1}{\ell^h}$. Furthermore the assumption $\Delta > (1+\ell^{h})^2$ implies that $\sqrt{\Delta}>1+\ell^{h}$. Putting these two inequalities together we have that
$$\frac{1}{\iota(\alpha_2)}=\iota(\alpha_1)- \lfloor\alpha_1 \rfloor_{\ell}\le \iota(\alpha_1) - \frac{1}{\ell^{h}}=\frac{1+\sqrt{\Delta}}{\Delta-1}-\frac{1}{\ell^{h}}=\frac{1}{\sqrt{\Delta}-1}-\frac{1}{\ell^{h}}<\frac{1}{\ell^{h}}-\frac{1}{\ell^{h}}=0,$$
so $\iota(\alpha_2)<0$. But also $\iota_-(\alpha_2)<0$, as $\iota_-(\alpha_1)<0$, and by Theorem~\ref{general_condition} this implies that the RCF expansion of $\delta$ is not periodic, as wanted.
\end{proof}


\subsection{Periodicity of the expansion for varying \texorpdfstring{$\ell$}{l}}

In this section we show that, for a fixed irrational quadratic number $\alpha$, there are at most finitely many primes $\ell$ and embeddings $\iota:\Q(\alpha)\to\Ql$ such that $\iota(\alpha)$ has a periodic RCF expansion, thus answering a question posed by Professor Corvaja. 

\begin{prop}\label{prop.quadratic.varyingl}
Let $\alpha\in\Ql$ be a root of the irreducible polynomial $Ax^2+Bx+C$, with  $A,B,C\in\Z, A>0$ and $\Delta:=B^2-4AC>0$. Assume that $\ell\nmid A$ and $\ell>\max\left(\frac{\Delta}{4A},C\right)$. Then the RCF expansion of $\alpha$ is not periodic.
\end{prop}

\begin{proof}
Because $\ell\nmid AC$, it is easy to see that $v(\alpha)=0$ and therefore $a:=\pfloor{\alpha}\in\Z$.
We have that $Aa^2+Ba+C\equiv 0 \pmod \ell$, because $a$ is the $\ell$-adic integral part of $\alpha$. The integral number $Aa^2+Ba+C$ cannot be zero, because the polynomial $Ax^2+Bx+C$ is irreducible and $a\in\Z$, therefore $\abs{Aa^2+Ba+C}\geq \ell$.

It is impossible that $Aa^2+Ba+C\leq -\ell$, because $\ell>\frac{\Delta}{4A}$.
The smallest solution of $Ax^2+Bx+C=\ell$ is negative, because $\ell>C$, but $a>0$ so the only possibility is that $a\geq \frac{-B+\sqrt{\Delta+4A\ell}}{2A}$, which is strictly bigger than both the real roots of $Ax^2+Bx+C$. But then both real embeddings of $\alpha_1=(\alpha-a)^{-1}$ are negative and Theorem~\ref{general_condition} implies that the RCF expansion of $\alpha$ is not periodic.
\end{proof}

\begin{cor}
Let $f(x)=Ax^2+Bx+C\in \Z[x]$ be an irreducible quadratic polynomial. Then there are at most finitely many primes $\ell$ such that there exists a root of $f(x)$ in $\Ql$ with a periodic RCF expansion.
\end{cor}
\begin{proof}
Without loss of generality, we assume $A>0$. Let $\ell$ be a prime number such that $f(x)$ has a root in $\Ql$. If $\Delta:=B^2-4AC<0$, then Theorem~\ref{T.periodic} guarantees that the RCF expansion of both roots of $f(x)$ in $\Ql$ is not periodic. The same holds by Proposition~\ref{prop.quadratic.varyingl} if $\Delta>0$ and $\ell$ is big enough. 
\end{proof}


\begin{section}{Algorithms} \label{algorithms}
In this section we collect the pseudo-code implementation of the decision algorithms that have been described along the paper.
\subsection{Rational numbers}
The first is an algorithm that decides in finite time whether the RCF expansion of a rational number is periodic or finite.
\begin{algorithm}[H]
  \caption{Deciding on the periodicity of the RCF expansion of a rational number}
  \label{Algo1}
  \begin{algorithmic}[1]
    \REQUIRE A rational number $\alpha=a/b$.
    \ENSURE The algorithm tells whether the RCF expansion of $\alpha$ is finite or periodic.
    \STATE $x:=\alpha$
    \STATE $B_1:=\max\left(\frac{\log b}{\log\ell},2\right)$
    \FOR{$i=1$ to $B_1$}
        \IF{$x<0$}
            \RETURN{The expansion is periodic.}
        \ENDIF
        \STATE $y:=x-\pfloor{x}$
        \IF{$y==0$}
            \RETURN{The expansion is finite.}
        \ENDIF
        \STATE $x:=1/y$
    \ENDFOR
   \end{algorithmic}
 \end{algorithm}

The second is an algorithm that computes in finite time the RCF expansion of a rational number.
\begin{algorithm}[H]
  \caption{Computing the RCF expansion of a rational number}
  \label{Algo2}
  \begin{algorithmic}[1]
    \REQUIRE A rational number $\alpha=a/b$.
    \ENSURE The algorithm outputs the RCF expansion of $\alpha$, divided into pre-periodic and periodic part if the expansion is not terminating.
    \STATE $x:=\alpha$
    \STATE $B_2:=32\ell H(\alpha)^2$
    \FOR{$i=1$ to $B_2+1$}
        \IF{$x==-\ell^{-1}$}
            \RETURN{$\overline{\ell-\ell^{-1}}$}
        \ENDIF
        \PRINT{$\pfloor{x}$}
        \STATE $y:=x-\pfloor{x}$ 
        \IF{$y==0$}
            \RETURN{}
        \ENDIF
        \STATE $x:=1/y$
    \ENDFOR
   \end{algorithmic}
 \end{algorithm}

Both algorithms are easy to describe and analyse. The first one executes the continued fraction algorithm until the expansion terminates or a negative complete quotient appears. Either one or the other of the stopping conditions will occur within $B_1=\max\left(\frac{\log b}{\log\ell},2\right)$ steps, according to the Quantified Algorithm (i) and to Remark~\ref{remark.rational}, while Theorem~\ref{T.terminating} ensures that the existence of a negative complete quotient implies the periodicity of the expansion. The running time is clearly bounded by $O\left(\frac{h(\alpha)}{\log\ell}\right)$ steps.

\bigskip

The second algorithm executes the same operations and prints the partial quotients of the expansion until the expansion terminates or a complete quotient equal to $-\ell^{-1}$ is found, at which point the expansion becomes periodic by Example~\ref{ex.l} and the algorithm prints the periodic part $\overline{\ell-\ell^{-1}}$. Theorem~\ref{T.terminating} again ensures that one of these two conditions will eventually occur and the Quantified Algorithm (ii) together with Remark~\ref{remark.rational} guarantees that one of the stopping conditions will occur within $B_2+1$ steps.
The running time of Algorithm~\ref{Algo2} can be then bounded by $O(\ell H(\alpha)^2)$.

\bigskip

Both Algorithm~\ref{Algo1} and \ref{Algo2} only need to store in memory the values of $x$ and $y$, which are updated at every step. We recall that the space required to store an algebraic number of height $\alpha$ is about $h(\alpha)$; then, the space complexity of both algorithms is bounded by $O\left(\max_{0\leq i\leq B_{1,2}} h(\alpha_i)\right)$. In particular, for the first algorithm using~\eqref{height_cq}  we get  $O\left(H(\alpha)^{\frac{\log 2}{\log\ell}}\log\left(\ell H(\alpha)\right)\right)$, while for the second algorithm the Quantified Algorithm (ii) gives $O\left(\ell H(\alpha)\right)$.


\subsection{Quadratic irrational numbers}
The third algorithm implements the second part of Proposition~\ref{Prop.Part.shape}, that is it computes the first steps of the RCF expansion of a quadratic irrational number until a complete quotient satisfying conditions~\eqref{Part.shape} is found.
\begin{algorithm}[H]
  \caption{The RCF expansion of a quadratic irrational number}
  \label{Algo3}
  \begin{algorithmic}[1]
    \REQUIRE A quadratic irrational number $\alpha$ with $v(\alpha)\leq 0$, represented by a 4-tuple $(f,c,b,\Delta)$.
    \ENSURE The algorithm outputs the first $M\leq \frac{v(\Delta)}{2}+2$ steps of the RCF expansion of $\alpha$, and a 4-tuple $(f_M,c_M,b_M,\tilde\Delta)$ representing $\alpha_M$ and satisfying~\eqref{Part.shape}.
    \STATE $d:=c/\gcd(c,b^2-\Delta)$
    \STATE $c:=cd$
    \STATE $b:=bd$
    \STATE $\Delta:= d^2\Delta$
    \STATE $h:=v(\Delta)/2$
    \FOR{$i=1$ to $h+2$}
        \STATE $\alpha:=\frac{b+\sqrt{\Delta}}{\ell^{f}c}$
        \STATE $a:=\pfloor{\alpha}$
        \PRINT{$a$}
        \STATE $b:=a \ell^{f} c - b$
        \STATE $c:=(\Delta-b^2)/(\ell^{v(\Delta-b^2)}c)$   
        \STATE $f:=v(\Delta-b^2)-f$
        \STATE $k:=\min(v(\Delta)/2,v(b))$
        \STATE $b:=b/\ell^k$
        \STATE $\Delta:=\Delta/\ell^{2k}$
        \STATE $f:=f-k$    
        \IF{conditions~\eqref{Part.shape} are satisfied}
            \RETURN{$\alpha_i$ is represented by $(f,c,b,\Delta)$}
        \ENDIF 
    \ENDFOR
   \end{algorithmic}
 \end{algorithm}
Lines 1--4 ensure that condition~\eqref{first.shape} is satisfied. Lines 7--12 perform one step of the expansion and the recurrence formulae~\eqref{quadratic.recurrence} and lines 13-16 simplify as many factors $\ell$ as possible.  Proposition~\ref{Prop.Part.shape} guarantees that at as long as $\Delta$ remains divisible by $\ell$ at least one factor $\ell$ is simplified at every step, and that when $\ell\nmid \Delta$ is satisfied, the remaining conditions are also satisfied within two more steps.

The time complexity is clearly bounded by $O(h(\alpha))$, which is $O\left(\frac{\log\Delta}{\log\ell}\right)$.

The algorithm stores in memory at any time only the data relative to one single step of the RCF expansion. A clear upper bound for the quantities $\Delta,\abs{b},\abs{c \ell^f}$ is given by the values of the recurrence sequences defined by~\eqref{quadratic.recurrence} without simplifying any factor $\ell$. These recurrence sequences are bounded in Proposition~\ref{P.bound.em}.

We have that the space complexity is bounded by $O\left(\max_{0\leq i\leq h+2} h(\alpha_i)\right)$. The height $h(\alpha_i)$ is $O(\log\Delta+\log\abs{b_i}+\log\abs{c_h\ell^{f_i}})$ so, according to Proposition~\ref{P.bound.em}, we have
\begin{align*}
O\left(\max_{0\leq i\leq h+2} h(\alpha_i)\right)&=    O\left(\log(\Delta)+h\log(\ceigen(\ell))+\log\ceigensnd(\alpha)\right)\\
&=O\left(\log\left (\Delta+b^2+c\ell^f \right)\right),
\end{align*}
because $\log(\ceigen(\ell))/\log\ell<3$ independently of $\ell$.

\bigskip

The last algorithm implements Theorem~\ref{general_condition}. It decides whether the RCF expansion of a quadratic irrational number is periodic, and in this case it computes the preperiodic and periodic parts of the expansion.
\begin{algorithm}[H]
  \caption{The RCF expansion of a quadratic irrational number}
  \label{Algo4}
  \begin{algorithmic}[1]
    \REQUIRE A quadratic irrational number $\alpha$, represented by a 4-tuple $(f,c,b,\Delta)$ which satisfies~\eqref{E.shape} and~\eqref{first.shape}.
    \ENSURE The algorithm outputs the RCF expansion of $\alpha$, divided into pre-periodic and periodic part, if the expansion is periodic; otherwise it tells that the expansion is aperiodic.
    \IF{$\Delta<0$}
        \RETURN{The expansion is not periodic.}
    \ENDIF
    \STATE $x_0:=\alpha$
    \STATE $t:=\floor{\sqrt{\Delta}}$
    \STATE $N:=\max(0,\abs{b}-t)+(2t+1)\Delta-\frac{t(t+1)(2t+1)}{3}+1$
    \FOR{$i=0$ to $N$}
        \IF{both real embeddings of $x_i$ are negative}
            \RETURN{The expansion is not periodic.}
        \ENDIF
        \STATE $y:=x_i-\pfloor{x_i}$
        \STATE $x_{i+1}:=1/y$
        \FOR{$j=0$ to $i$}
            \IF{$x_j==x_{i+1}$}
                \PRINT{Preperiodic part:}
                \FOR{$k=0$ to $j-1$}
                    \PRINT{$\pfloor{x_k}$}
                \ENDFOR
                \PRINT{Periodic part:}
                \FOR{$k=j$ to $i$}
                    \PRINT{$\pfloor{x_k}$}
                \ENDFOR
                \RETURN
            \ENDIF
        \ENDFOR
    \ENDFOR
   \end{algorithmic}
 \end{algorithm}
 The algorithm simply executes the RCF iterations and at each step it compares the new complete quotient with all the previous ones until a repetition is detected or a complete quotient with two negative embeddings is reached. Thanks to Theorem~\ref{general_condition} we know that in at most $N$ steps one of these conditions will occur.

Due to the nested iterations, the time complexity is $O(N^2)$, which is $O\left(\Delta^3+b^2\right)$.
 
 Unlike the previous algorithms, this one needs to store in memory the whole sequence of complete quotients in order to detect repetitions. The space complexity is therefore bounded by $O\left(\sum_{i=0}^{N}h(\alpha_i)\right)$.

By~\eqref{height_cq2} we have that $h(\alpha_i)\leq h(\alpha)+s_i\log\ell+i\log(2\ell)$. As shown in the proof of Proposition~\ref{P.bound.em}  we have that $e_i\leq f_i<3i+\ceigentrd(\alpha)$, so that $s_i=O(i^2+i\ceigentrd(\alpha))$ and
\[
O\left(\sum_{i=0}^{N}h(\alpha_i)\right)=O\left(N h(\alpha)+ N^3 \log\ell+N^2 \ceigentrd(\alpha)\log\ell+N^2\log(2\ell)\right).
\]

We already remarked analysing the previous algorithm that $$h(\alpha)=O\left(\log\Delta+\log\abs{b}+\log\abs{c\ell^{f}}\right),$$ while $\ceigentrd(\alpha)=O(\log(\Delta+b^2+\abs{c\ell^f}))$, as computed in Proposition~\ref{P.bound.em}.

In the end, the space complexity of the algorithm is bounded by
\[
O\left(\left(\Delta^\frac{3}{2}+\abs{b}\right)^3 \log\ell \left(\log\Delta+\log\abs{b}+\log\abs{c\ell^f}\right)\right).
\]

\begin{rem}
   As noticed in Remark~\ref{explicit_constant}, if $\Delta= \ell^h \tilde \Delta$ for some $h\ge 1$ and $(\tilde \Delta, \ell)=1$, it is more convenient to apply first Algorithm~\ref{Algo3} to reduce $\alpha$ to the form~\eqref{Part.shape}, and then to apply Algorithm~\ref{Algo4} to the reduced form $(f_M, c_M, b_M, \tilde \Delta)$. Using the improved estimate for $N$ that comes out of Remark~\ref{explicit_constant}, we have that the total complexity is bounded by $O\left (h^2+ \tilde \Delta^3 + b^2  \right )$ in place of $O(\Delta^3+b^2)$. 
\end{rem}
\end{section}


\section*{Acknowledgments}
The authors would like to thank the anonymous referee for several helpful comments and remarks, Professor Shigeki Akiyama for suggesting useful references, Gabriel Dill for useful discussions and comments and the Scuola Normale Su\-pe\-rio\-re for support. The first author was funded by the INdAM [Borsa Ing. G. Schirillo], the European Research Council [267273] and the the Engineering and Physical Sciences Research Council [EP/N007956/1]. The second author thanks the University of Basel and the Centro di Ricerca Matematica Ennio de Giorgi for support.


\bibliographystyle{amsalpha}
\bibliography{biblio}

\end{document}